\documentclass[10pt]{amsart}
\usepackage{amssymb}
\usepackage{tikz-cd}

\title[Amenability, Optimal Transport and Ergodic Theorems]{Amenability, Optimal Transport and Abstract Ergodic Theorems}

\author {Christian Rosendal}
\address{Department of Mathematics\\University of Maryland\\4176 Campus Drive - William E. Kirwan Hall\\College Park, MD 20742-4015\\USA}
\email{rosendal@umd.edu}
\urladdr{sites.google.com/view/christian-rosendal/}

\newcommand{\norm}[1]{\lVert#1\rVert}
\newcommand{\Norm}[1]{\big\lVert#1\big\rVert}
\newcommand{\NORM}[1]{\Big\lVert#1\Big\rVert}

\newcommand{\forkindep}[1][]{\mathop{\mathop{\vcenter{\hbox{\oalign{\noalign{\kern-.3ex}
\hfil$\vert$\hfil\cr\noalign{\kern-.7ex}$\smile$\cr\noalign{\kern-.3ex}}}}}\displaylimits_{#1}}}

\newcommand{\maps}[1]{\mathop{\overset{#1}\longrightarrow}}

\newcommand{\acts}[1]{\mathop{\overset{#1}\curvearrowright}}

\newcommand{\Mgd}[2]{\big\{  {#1}\;\big|\; {#2} \big\} }

\newcommand{\maths}[1]{\[\begin{split}{#1}\end{split}\]}

\newcommand {\Z}{\mathbb Z}
\newcommand {\Q}{\mathbb Q}
\newcommand {\R}{\mathbb R}

\newcommand {\M}{\mathbb M}

\newcommand{\eps}{\epsilon}

\newcommand{\iso}{\cong}

\newcommand{\tom} {\emptyset}

\newcommand{\begr}{\!\upharpoonright}

\newcommand{\equi}{\Leftrightarrow}

\newcommand{\Lim}[1]{\mathop{\longrightarrow}\limits_{#1}}

\newcommand {\Del}{ \; \Big| \;}
\newcommand {\del}{ \; \big| \;}

\newcommand {\go} {\mathfrak}
\newcommand {\ku} {\mathcal}

\newcommand{\ov}{\overline}
\newcommand{\inv}{^{-1}}

\renewcommand {\a} {\forall}

\newcommand {\spa} {\sf span}

\theoremstyle{plain}
\newtheorem{thm}{Theorem}[section]
\newtheorem{cor}[thm]{Corollary}
\newtheorem{lemme}[thm]{Lemma}
\newtheorem{prop} [thm] {Proposition}
\newtheorem{defi} [thm] {Definition}

\newtheorem{prob}[thm]{Problem}

\theoremstyle{definition}

\newtheorem{rem}[thm]{Remark}
\newtheorem{exa}[thm]{Example}

\definecolor{groen}{rgb}{0,0.5,.7}
\definecolor{gul}{rgb}{0.94,0.8,0}
\definecolor{blaa}{rgb}{0.16,0,0.6}
\definecolor{roed}{rgb}{1,0,0}

\makeindex

\begin{document}

\thanks{The author was partially supported by the NSF through  awards DMS 2204849 and DMS 2246986.}

\maketitle

\begin{abstract}
Using tools from the theory of optimal transport, we establish several results concerning isometric actions of amenable topological groups with potentially unbounded orbits. Specifically, suppose $d$ is a compatible left-invariant metric on an amenable topological group $G$ with no non-trivial homomorphisms to $\mathbb R$. Then, for every finite subset $E\subseteq G$ and $\epsilon>0$, there is a finitely supported probability measure $\beta$ on $G$ such that
$$
\max_{g,h\in E}\, {\sf W}(\beta g, \beta h)<\eps,
$$
where ${\sf W}$ denotes the Wasserstein distance between probability measures on the metric space $(G,d)$. When $d$ is the word metric on a finitely generated group $G$, this strengthens a well known theorem of Rei\-ter \cite{reiter} and, when $d$ is bounded, recovers a result of Schneider--Thom \cite{thom}. Furthermore, when $G$ is locally compact, $\beta$ may be replaced by an appropriate probability density $f\in L^1(G)$.

Also, when $G\curvearrowright X$ is a continuous isometric action on a metric space, the space of Lipschitz functions on the quotient $X/\!\!/G$ is isometrically isomorphic to a $1$-complemented subspace of the Lipschitz functions on $X$. And, when additionally $G$ is skew-amenable, there is a $G$-invariant  contraction 
$$
\go {Lip}\, X \maps S\go{Lip}(X/\!\!/G)
$$
so that $(S\phi\big)\big(\ov{Gx}\big)=\phi(x)$ whenever $\phi$ is constant on every orbit of $G\curvearrowright X$. This latter extends results of C\'uth--Doucha \cite{cuth} from the setting of locally compact or balanced groups.
\end{abstract}

\tableofcontents

\section{Introduction}

Recently, in connection with the general study of Polish topological groups, there has been a renewed interest in the concept of amenability for non-locally compact topological groups (see, for  example, the papers \cite{cuth, juschenko, moore, thom2,thom}). The present paper is a further contribution to that direction of research and aims, on the one hand, to extend the results of \cite{cuth} and \cite{thom} by considering both general amenable topological groups and, more  importantly, allowing metrics of potentially infinite diameter.  In particular, this leads us to an improvement of H. Reiter's well-known criterion for amenability \cite{reiter} even in the case of finitely generated groups (see Corollary \ref{intro:reiter} below).

Our first main result provides an extension of a recent characterisation of amena\-bi\-lity for general topological groups  due to F. M. Schneider and A. Thom \cite{thom}. Two issues come into play here. First of all, we reformulate the result as a statement about optimal transport in topological groups. This allows us to introduce the machinery of {\em Arens--Eells} or {\em transportation cost} spaces and to apply the needed functional analytical machinery to extend the characterisation to potentially unbounded \'ecarts.

Recall first that an {\em \'ecart} or {\em pseudo-metric} $d$ is a metric except that one may have $d(x,y)=0$ for distinct points $x,y$. Also, for a group $G$, let $\R G$ denote the real group algebra of $G$ and $\Delta G$ the convex hull of $G$ in $\R G$. If then $d$ is an \'ecart on $G$, we can define the {\em Arens--Eells norm}\footnote{
Depending on the perspective, this is also known as the {\em Kantorovich--Rubinstein norm}, {\em Wasserstein}, {\em optimal transport} or {\em earth movers distance}. See further comments in Sections \ref{sec:Lipschitz} and \ref{sec:KR}.
}  
$\norm{\alpha-\beta}_{\AE}$ between $\alpha,\beta\in \Delta G$ as the cost of an optimal transport between $\alpha$, viewed as a collection of manufacturers whose total production is one unit mass, and $\beta$, viewed as a collection of consumers with total consumption one, and where the cost of transporting one unit mass from $g\in G$ to $f\in G$ is just $d(g,f)$. That is,
$$
\norm{\alpha-\beta}_{\AE}=\inf\Big(\sum_{i=1}^n|a_i|d(g_i,f_i)\Del \alpha-\beta=\sum_{i=1}^na_i(g_i-f_i)\Big).
$$


\begin{thm}\label{intro:amen-char}
Suppose $d$ is a continuous left-invariant \'ecart on an amenable topological group $G$ and assume that either $d$ is bounded  or that $G$ has no non-trivial continuous homomorphisms to $\R$. 
Then, for every finite subset $E\subseteq G$, we have 
$$
\inf_{\beta\in \Delta G}\;\;\max_{g,f\in E}\;\;\norm{\beta g-\beta f}_{\AE}=0.
$$
\end{thm}

Remarkably, Theorem \ref{intro:amen-char} appears to provide new information even about discrete amenable groups. Unlike Reiter's criterion for amenability \cite{reiter}, we are able to bring the large scale geometry of finitely generated groups into play. For this, recall that, if $S\subseteq G$ is a finite generating set for a group $G$, the associated word metric is defined by
$$
d_S(g,f)=\min(k\del g=fs_1s_2\cdots s_k \text{ for some }s_i\in S^\pm).
$$
As is easy to see, any other choice of $S$ results in a word metric bi-Lipschitz equivalent to the one above. In particular, the choice of specific finite generating set is immaterial for Corollary \ref{intro:reiter} below.

Observe that elements of $\Delta G$ may alternatively be viewed as positive elements of the unit sphere of the Banach space $\ell^1(G)$. Furthermore, as the minimal positive distance of the word metric $d_S$ is $1$, we find that
$$
\norm{\alpha-\beta}_{\ell^1(G)}\leqslant 2\,\norm{\alpha-\beta}_{\AE}
$$
for all $\alpha,\beta\in \Delta G$. Reiter's criterion applied to a finitely generated group $G$ then states that
$$
\inf_{\beta\in \Delta G}\;\;\max_{g,f\in E}\;\;\norm{\beta g-\beta f}_{\ell^1(G)}=0
$$
for any finite subset $E\subseteq G$ and thus Theorem \ref{intro:amen-char} is seen to be a strengthening of Reiter's criterion in this setting.  With a little more work, we obtain the following.

\begin{cor}\label{intro:reiter}
Under the assumptions of Theorem \ref{intro:amen-char}, for all finite subsets $E\subseteq G$ and $\eps>0$, 
there are $h_1,\ldots, h_n\in  G$ so that
$$
\min_{\sigma\in {\sf Sym}(n)}\;\;\frac1n \sum_{i=1}^nd(h_ig,h_{\sigma(i)})<\eps
$$
for all $g\in E$.
\end{cor}
In particular, Corollary \ref{intro:reiter} applies to discrete groups such as the infinite dihedral group $D_\infty=(\Z/2\Z)*(\Z/2\Z)$, but fails for the case of $\Z$ and $\R$.

Theorem \ref{intro:amen-char} has a version more suitable for locally compact groups,second-countable formulated in terms of probability densities on $G$. Namely, let $G$ be a locally compact second-countable group equipped with its left Haar measure. Suppose also that $d$ is a compatible left-invariant metric on $G$ and consider the space $L^1_{d,+,1}(G)$ of probability densities on $G$, that is, non-negative $f\in L^1(G)$ with $\norm f_{L^1}=1$, satisfying the additional assumption 
$$
\int d(x,1)f(x)\, dx<\infty.
$$
We define the {\em Wasserstein distance} between $f,h\in L^1_{d,+,1}(G)$ to be
\maths{
{\sf W}(f,h)=&\inf_{\eta} \int d(x,y)\, d\eta(x,y)\\
=&\inf_{({\sf X},{\sf Y})}\mathbb E \,d({\sf X}, {\sf Y}),
}
where the infimum is taken over all probability measures $\eta$ on $G\times G$ whose marginals are $f\, dx$ and $h\, dx$,  respectively over all pairs of $G$-valued random variables ${\sf X}$ and ${\sf Y}$ with densities $f$ and $h$. Thus ${\sf W}(f,h)$ measures the cost of an optimal transport of the probability density $f$ to $h$.

\begin{thm}\label{intro:loccomp}
Suppose $G$ is an amenable locally compact second-countable group and $d$ is a compatible left-invariant metric on $G$. Assume also that $G$ has no non-trivial continuous homomorphisms $G\to \R$. Then, for every compact subset $C\subseteq G$ and $\eps>0$, there is a compactly supported $f\in L^1_{d,+,1}(G)$ so that 
$$
{\sf W}(R_yf,R_zf)<\eps
$$
for all $y,z\in C$.
\end{thm}

When $X$ is a metric space, let ${\sf Lip}\, X$ denote the vector space of all real-valued Lipschitz functions on $X$ and let 
$$
\go{Lip}\,X={\sf Lip}\, X/\{\text{constant functions}\}.
$$
Then $\go{Lip}\,X$ is a Banach space when equipped with the norm $L(\cdot)$ measuring the optimal Lipschitz constant of a function. Alternatively, $\go{Lip}\,X$ can be identified with the space ${\sf Lip}_0X$ of Lipschitz functions  taking the value $0$ at some specified point $p\in X$.

Our next main result generalises results of M. C\'uth and M. Doucha \cite{cuth} from locally compact to general amenable groups and to isometric actions with potentially unbounded orbits. However, in order to do this, one must exclude groups such as $\bigoplus_{n=1}^\infty\Z$ and $\prod_{n=1}^\infty\R$ with too many continuous homomorphisms to $\R$. So, for a topological group $G$, let ${\sf Hom}(G,\R)$ denote the vector space of continuous homomorphisms from  $G$ to $\R$. In Theorem \ref{intro:cuth}, we then have to assume that ${\sf Hom}(G,\R)$ has finite dimension. For example, this happens when $G$ is {\em topologically finitely generated}, that is, contains a finitely generated dense subgroup. Similarly, by considering the associated homo\-morphism between the respective Lie algebras, this also applies to all connected locally compact Lie groups such as $G=\R^n$. 
A slightly more refined statement dealing with groups for which ${\sf Hom}(G,\R)$ may be infinite-dimensional can be found in Theorem \ref{thm:cuth1}.

\begin{thm}\label{intro:cuth}
Suppose $G$ is an amenable topological group and $G\curvearrowright X$ is a continuous isometric action on a metric space. Assume also that the vector space ${\sf Hom}(G,\R)$ of continuous homomorphisms from $G$ to $\R$ is finite-dimensional.
Then the closed linear subspace 
$$
\{\phi\in \go{Lip}\, X\del \phi\text{ is constant on every orbit of }G\curvearrowright X\}
$$
is complemented in  $\go {Lip}\, X$. 
\end{thm}

When $G\curvearrowright X$ is an isometric action on a metric space, let $X/\!\!/G$ denote the collection of closures of $G$-orbits on $X$. Then $X/\!\!/G$  is a metric space when equipped with the Hausdorff distance,
$$
d_H\big(\ov{Gx},\ov{Gy}\big)=\max\big\{\sup_{z\in \ov{Gx}}{\sf dist}(z, \ov{Gy}),\sup_{z\in \ov{Gy}}{\sf dist}(z, \ov{Gx})\big\}.
$$
As is easy to verify, $\go{Lip}(X/\!\!/G)$ is isometrically isomorphic to the space $\{\phi\in \go{Lip}\, X\del \phi\text{ is constant on every $G$-orbit}\}$ and we thus get the following corollary.
\begin{cor}
Under the assumptions of Theorem \ref{intro:cuth}, $\go{Lip}(X/\!\!/G)$ is isometrically isomorphic to a complemented subspace of $\go{Lip}\, X$.
\end{cor}

The essence of the above results can be understood in terms of integration. Namely, we aim to find a way to integrate potentially unbounded Lipschitz functions defined on metric spaces in a $G$-invariant way. However, for this we need to assume that the group is not only amenable, but also skew-amenable, a concept that is properly introduced in Section \ref{sec:skew}.

\begin{thm}\label{intro:equivproj}
Suppose $G$ is an amenable and skew-amenable topological group and that $G\curvearrowright X$ is a continuous isometric action on a metric space $X$. Assume furthermore that either $X$ has finite diameter or that $G$ has no non-trivial continuous homomorphisms to $\R$. 
Then there is a $G$-invariant  contraction 
$$
\go {Lip}\, X \maps S\go{Lip}(X/\!\!/G)
$$
so that $(S\phi\big)\big(\ov{Gx}\big)=\phi(x)$ whenever $\phi$ is constant on every orbit of $G\curvearrowright X$.
\end{thm}

The operator $S$ should be understood as integrating a Lipschitz function $\phi$ over every $G$-orbit closure $\ov{Gx}$ with output value $S\phi(x)$.

For our last set of results, let us recall that any continuous action $G\overset\alpha\curvearrowright V$ of a topological group $G$ by affine isometries on a Banach space $V$ decomposes uniquely into a strongly continuous  linear isometric action $G\overset\pi\curvearrowright V$ and an associated cocycle $G\overset b\longrightarrow V$, in the sense that
$$
\alpha(g)v=\pi(g)v+b(g)
$$
for all $g\in G$ and $v\in V$. An important theme in geometric group theory (cf., properties (T), (FH) and the Haagerup property) is to understand fixed point properties of such actions under varying geometric and analytic  assumptions on $V$ or on the group $G$. Theorem \ref{intro:affine} below establishes what can be said about such actions of amenable groups without any assumptions on the underlying Banach space $V$.

\begin{thm}\label{intro:affine}
Suppose $G$ is an amenable topological group with no non-trivial continuous homomorphisms $G\to \R$ and let $G\overset\alpha\curvearrowright V$ be a continuous affine isometric action on a Banach space with associated linear isometric action $G\overset\pi\curvearrowright V$. Then
$$
V_G=\ov{\sf span}\{v-\pi(g)v\del v\in V\;\&\; g\in G\}=\{v\in V\del 0\in \ov{\pi(\Delta G)v}\}
$$
is a closed linear subspace of $V$ invariant under the affine action $G\overset\alpha\curvearrowright V$.
\end{thm}

Whereas this does not itself establish any fixed points, it does suffice to prove that the convex $\alpha$-invariant subsets of $V_G$ are almost directed under inclusion.

\begin{cor}
For any finite collection $C_1,C_2,\ldots, C_n$ of non-empty convex $\alpha$-invariant subsets of $V_G$ and any $\eps>0$, there is some $v\in V_G$ so that 
$$
\max_{i=1}^n\;{\sf dist}(v,C_i)<\eps.
$$
\end{cor}

The organisation of the paper is as follows. In Sections \ref{sec:algebra}, \ref{sec:Lipschitz} and \ref{sec:KR}, we fix the notation regarding group algebras, amenability, Arens--Eells and Lipschitz spaces, Kantorovich duality and provide a summary of the background material needed later on. In Sections \ref{mean-operators}, \ref{decompo} and \ref{sec:skew}, we provide some general results for mean operators associated with linear actions of amenable and skew-amenable groups.  Section \ref{sec:affine} focuses on approximate fixed points and cohomology of affine isometric actions of such groups, whereas, in Section \ref{isom}, we focus on isometric actions on metric spaces and present the fundamental lemmas needed to handle unbounded orbits. In Sections \ref{amen gps} and \ref{LCSC}, this is put into play in order to establish  the main results. 

\textsc{Acknowledgement:} The author is grateful for interesting feedback from M. C\'uth, N. Monod and F. M. Schneider and for multiple interesting conversations with  T. Tsankov on the topic of the paper.


\section{Group algebra, Banach modules and amenability}\label{sec:algebra}
If $X$ is any set, we let $\R X$ be the free vector space over $X$, that is, the set of finitely supported functions $\xi \colon X\to \R$ or, equivalently, formal linear combinations
$$
\xi=\sum_{i=1}^na_ix_i
$$
where $x_i\in X$ and $a_i\in \R$.
We also let $\M X$ be the hyperplane in $\R X$ consisting of all $\xi\in \R X$ of {\em mean $0$}, that is,  such that $\sum_{x\in X}\xi(x)=0$, and note that 
$$
\M X={\sf span}\{x-y\del x,y\in X\}.
$$
Finally, let $\Delta X$ be the {\em standard simplex} in $\R X$, i.e., the collection of all finite convex combinations $\beta=\sum_{i=1}^n\lambda_ix_i$ of points $x_i\in X$.

Note that, when $G$ is a group, elements of $\R G$ multiply as follows
$$
\Big(\sum_{i=1}^na_ig_i\Big)
\cdot
\Big(\sum_{j=1}^mb_jf_j\Big)
=\sum_{i,j}a_ib_j{g_if_j}.
$$
In other words, viewing $\xi=\sum_{i=1}^na_ig_i$ and $\zeta=\sum_{j=1}^mb_jf_j$ as finitely supported functions $G\overset{\xi,\zeta}\longrightarrow \R$, the product $\xi\cdot \zeta$ is just the convolution $\xi\ast \zeta$ defined by
$$
(\xi\ast \zeta)(h)=\sum_{g\in G}\xi(g)\zeta(g\inv h).
$$
The resulting algebra $\R G$ is called the {\em group algebra} of $G$. Note also that the simplex $\Delta G$ is a subsemigroup of the semigroup $(\R G, *)$.

If $G$ is a topological group, a {\em Banach $G$-module} is a pair $(V,\pi)$, where $V$ is a Banach space and $G\overset\pi\curvearrowright V$ is a continuous linear action of $G$ on $V$. Alternatively, we may view $\pi$ as a group representation $G\overset\pi\longrightarrow \ku L(V)$  of $G$ in the algebra $\ku L(V)$ of bounded linear operators on $V$ such that $\norm{\pi(\cdot)}$ is bounded on a neighbourhood of the identity in $G$ and such that $g\mapsto \pi(g)v$ is continuous for every $v\in V$ (i.e., the representation is continuous with respect to the strong operator topology on $\ku L(V)$). In case every $\pi(g)$ is an isometry of $V$, we say that $(V,\pi)$ is an {\em isometric} Banach $G$-module. Observe also that the group representation extends uniquely to a representation $\R G\overset\pi\longrightarrow \ku L(V)$ of the group algebra. 

If $G\overset\pi\curvearrowright V$ is an isometric linear action, we define the {\em contragredient} or {\em dual action} of $G\overset{\pi^*}\curvearrowright V^*$ simply by $\pi^*(g)\phi=\phi\circ \pi(g\inv)$, that is, $\pi^*(g)=\pi(g\inv)^*$.  Even when the action $G\overset\pi\curvearrowright V$ is continuous, $(V^*,\pi^*)$ will not, in general, be a Banach $G$-module because the action may no longer be continuous with respect to the norm topology on $V^*$. So one must exercise some care when dealing with these dual actions.

If $G$ is a topological group, then $G$ acts by isometric automorphisms  on the Banach algebra $\ell^\infty(G)$ of bounded real-valued functions on $G$ by
$$
\rho(g)(\phi)=\phi(\,\cdot\, g)
$$
and similarly by
$$
\lambda(g)(\phi)=\phi(g\inv \,\cdot\,).
$$
Both of these action fail, in general, to be continuous, so one may restrict the attention to the set ${\sf LUCB}(G)$ of elements $\phi\in \ell^\infty(G)$ such that  the evaluation map
$$
g\in G\mapsto \rho(g)\phi\in \ell^\infty(G)
$$
is continuous. Because the action $G\overset{\rho}\curvearrowright \ell^\infty(G)$ is by isometric algebra automorphisms, we note that ${\sf LUCB}(G)$ is a closed $\rho$-invariant subalgebra of $\ell^\infty(G)$ and that the action $G\overset\rho\curvearrowright {\sf LUCB}(G)$ is continuous. The elements of ${\sf LUCB}(G)$ are the {\em bounded left-uniformly continuous} functions on $G$ and may equivalently be described as those $\phi\in \ell^\infty(G)$ such that, for every $\eps>0$, there is some identity neighbourhood $V\subseteq G$ such that
$$
\sup_{g\in G}\;\sup_{v\in V}\;\norm{\phi(gv)-\phi(g)}<\eps.
$$
Because the actions $\rho$ and $\lambda$ commute, one also notes that ${\sf LUCB}(G)$ is $\lambda$-invariant, though the $\lambda$-action on ${\sf LUCB}(G)$ is not, in general, continuous.

Similarly, by ${\sf RUCB}(G)$ we denote the collection of all $\phi\in \ell^\infty(G)$ such that
$$
g\in G\mapsto \lambda(g)\phi\in \ell^\infty(G)
$$
is continuous or, equivalently, such that, for every $\eps>0$, there is some identity neighbourhood $V\subseteq G$ such that
$$
\sup_{g\in G}\;\sup_{v\in V}\;\norm{\phi(vg)-\phi(g)}<\eps.
$$
Again, ${\sf RUCB}(G)$ is a $\lambda$ and $\rho$-invariant closed subalgebra of $\ell^\infty(G)$ and the $\lambda$-action is continuous.

Finally, the group $G$ is called {\em amenable} if every continuous affine action of $G$ on a compact convex subset of a locally convex topological vector space has a fixed point. By a result of N. W. Rickert (Theorem 4.2 \cite{rickert}),  $G$ is amenable if and only if there is a $\rho$-invariant mean on ${\sf LUCB}(G)$, that is, a positive linear functional $\go m\colon {\sf LUCB}(G)\to \R$ such that $\norm{\go m}=1$, $\go m({\bf 1})=1$ and $\go m\big(\rho(g)\phi)=\go m(\phi)$ for all $g\in G$ and $\phi\in {\sf LUCB}(G)$. 

On the other hand, $G$ will be said to be {\em skew-amenable} if there is a $\lambda$-invariant mean on ${\sf LUCB}(G)$. We shall return to this concept in  Section \ref{sec:skew}.


\section{Arens--Eells spaces and their duals }\label{sec:Lipschitz}
Assume henceforth that  $(X,d)$ is a metric space. We equip $\M X$ with the {\em Arens--Eells} or {\em transportation cost norm}\footnote{
As we shall see below, there are two simultaneous competing sources of this norm. One is due to R. Arens and J. Eells \cite{arens}, which tends to focus on the metric and Banach space theoretical properties and which, in conjunction with the paper by K. de Leuw \cite{deleuw}, seem to have mainly led to the  Banach space theoretical study of spaces of Lipschitz functions, see N. Weaver \cite{weaver}.  The other source are the foundational papers by L. V. Kantorovich and G. Sh. Rubinstein \cite{kantorovich1,kantorovich2} that instead have mainly influenced the modern theory of optimal transport (see L. V. Kantorovich and G. P. Akilov \cite{kantorovich3} and C. Villani \cite{villani}).}

\begin{equation}
\norm{\xi}_{\AE}=\inf\Big( \sum_{i=1}^n|a_i|d(x_i,y_i)\del \xi=\sum_{i=1}^na_i({x_i-y_i})\Big).
\end{equation}
Note that the norm $\norm{\xi}_{\AE}$ can be seen as a measure of the cost of an optimal transport between the {\em sources} or {\em manufacturers}
$$
\{x\in X\mid \xi(x)>0\}
$$
and {\em sinks} or {\em consumers}
$$
\{x\in X\mid \xi(x)<0\},
$$
where the cost of transporting {\em mass} $a$ from $x$ to $y$, represented by the term $a(x-y)$, is just $|a|d(x,y)$.

Observe that by replacing summands $a_i(x_i-y_i)$ by $-a_i(y_i-x_i)$, we may suppose that all  coefficients $a_i$ are non-negative. Also, a simple argument using the triangle inequality for $d$ (by cutting out the {\em middle man}), shows that we may suppose that no point appears both as a source, i.e., among the $x_i$, and as a sink, i.e., among the $y_i$.  It thus follows that
\begin{equation}\label{AE2}
\norm{\xi}_{\AE}=\inf\Big( \sum_{i=1}^na_id(x_i,y_i)\;\Big|\; \xi=\sum_{i=1}^na_i({x_i-y_i}), \;\; a_i>0, \;\; \xi(y_i)<0<\xi(x_i)\Big).
\end{equation}

We now let $\AE X$ be the Banach space completion of $\M X$ with respect to the Arens--Eells norm $\norm\cdot_{\AE}$. Let also ${\sf Lip}\, X$ denote the vector space of real-valued Lipschitz functions on $X$ equipped with the seminorm 
$$
L(\phi)=\sup_{x\neq y}\frac{|\phi x-\phi y|}{d(x,y)}
$$
that measures the optimal Lipschitz constant for $\phi$. 

Observe that every real-valued function $\phi$ on $X$ extends uniquely to a linear map $\hat \phi\colon \R X\to \R$ and thus, in particular, defines a linear functional on $\M X$. Furthermore, for all $\xi=\sum_{i=1}^na_i(x_i-y_i)\in \M X$,
\maths{
\big|\hat\phi(\xi)\big|
&= \Big|\sum_{i=1}^na_i( \phi x_i- \phi y_i)\Big|\\
&\leqslant \sum_{i=1}^n|a_i|\big| \phi x_i- \phi y_i\big|\\
&\leqslant L(\phi)\sum_{i=1}^n|a_i|d(x_i,y_i)
}
and so $\big|\hat\phi(\xi)\big|\leqslant L(\phi)\norm\xi_{\AE}$, showing that $\norm {\hat\phi}_{\AE}\leqslant L(\phi)$. Conversely, suppose  $\psi\in \M X^*$ and let $e\in X$ be any point. Define a real-valued Lipschitz map $\phi$ on $X$ by $\phi x=\psi(x-e)$ and note that $\psi$ and $\hat \phi$ agree on $\M X$. Moreover, as clearly $\norm{x-y}_{\AE}\leqslant d(x,y)$, we find that
$$
L(\phi)
=\sup_{x\neq y}\frac{|\phi x-\phi y|}{d(x,y)}
=\sup_{x\neq y}\frac{|\psi(x-y)|}{d(x,y)}
\leqslant \sup_{x\neq y} \frac{\norm{x-y}_{\AE}\cdot\norm{\psi}_{\AE}}{d(x,y)}
\leqslant \norm{\psi}_{\AE}.
$$
Thus, $\phi\mapsto \hat \phi$ is a linear operator from ${\sf Lip}\, X$ to $\M X^*$ such that $\norm{\hat \phi}_{\AE}\leqslant L(\phi)$. Furthermore, every $\psi\in \M X^*$ can be written as $\psi=\hat\phi$ where $L(\phi)\leqslant \norm{\psi}_{\AE}$. Note also that $\norm{\hat \phi}_{\AE}=0$ only if $\phi$ is constant. It therefore follows that  $\phi\mapsto \hat \phi$ defines a surjective linear isometry from the Banach space
$$
\go {Lip}\,X={\sf Lip}\, X/\{\text{constant functions}\}
$$
onto  the dual space $\M X^*=\AE X^*$. Henceforth, we will simply identify $\AE X^*$ with $\go {Lip}\,X$ and also identify the elements of $\go {Lip}\,X$ with their representatives in ${\sf Lip}\, X$. 

Note also that, because $\phi(\xi)=\sum_{x\in X}\xi(x)\phi(x)$, the norm on $\M X$ may alternatively be computed by
\begin{equation}\label{AE}
\norm{\xi}_{\AE}=\sup\big( \sum_{x\in X}\xi(x)\phi(x)\del \phi\colon X\to \R \text{ is $1$-Lipschitz }\big).
\end{equation}
The formula (\ref{AE}) is oftentimes called {\em Kantorovich duality}.

There is an instructive and well-known interpretation of this duality. Namely, suppose a bicycle manufacturer V\'elo Sportif operates several factories. To cover their energy needs, V\'elo Sportif has invested in a few offshore windmills whose total electricity output exactly matches the consumption at their factories. Their electricity production and consumption can then be represented by $\xi \in \M X$, where $X$ is the space of geographical locations. Transporting 1 MWh of electricity from location $x$ to $y$ nevertheless has a cost of $d(x,y)$, where $d$ can be assumed to be a metric on $X$. The goal of V\'elo Sportif is thus to {\em minimise the total cost}  by finding an optimal transportation cost plan
$$
\xi=\sum_{i=1}^na_i(x_i-y_i),
$$
that is, such that $\norm\xi_{\AE}=\sum_{i=1}^n|a_i|d(x_i,y_i)$.

A utility company \'Electricit\'e  Nucl\'eaire operating nuclear power plants however sees an opportunity for partnership. They offer V\'elo Sportif to purchase their entire windmill electricity production at a fixed market price of $\phi(x)$ per MWh dependent on the location $x\in X$ of the windmills (perhaps to provide to \'Electricit\'e  Nucl\'eaire's own consumers) and, on the other hand, sell back electricity from their nuclear power plants to V\'elo Sportif at the location $y\in X$ of their factories for the same market price $\phi(y)$. In this way, V\'elo Sportif will not pay any transportation cost. Of course, for this to make economical sense for V\'elo Sportif,  \'Electricit\'e  Nucl\'eaire guarantees that any cost differential $\phi(y)-\phi(x)$  does not exceed the transportation cost $d(x,y)$.
The total income of \'Electricit\'e  Nucl\'eaire from this partnership will then be 
$$
\text{Income}(\phi)=-\sum_{x\in X}\xi(x)\phi(x).
$$
The economists at \'Electricit\'e  Nucl\'eaire now aim to {\em optimise the market price} $\phi$ subject to the constraints
$$
\phi(y)-\phi(x)\leqslant d(x,y), \quad x,y\in X
$$
so as to maximise their income. Observe that the net savings of V\'elo Sportif from this partnership will be 
$$
\text{Savings}(\phi)=\norm\xi_{\AE}-\text{Income}(\phi).
$$

However, by Kantorovich--Rubenstein duality, \'Electricit\'e  Nucl\'eaire may invoke a $1$-Lipschitz function $\phi$ realising the supremum in (\ref{AE}), which means that the price constraints are satisfied, but nevertheless $\text{Savings}(\phi)=0$.

We note a well-known fact about integral consumption schemes.

\begin{lemme}\label{konig}
Let $x_1,\ldots, x_n$ and $y_1,\ldots, y_n$ be points  in a metric space $(X,d)$. Then
$$
\NORM{\sum_{i=1}^n{x_i}-\sum_{i=1}^n{y_i}}_{\AE}=\sum_{i=1}^nd(x_i,y_{\sigma(i)})
$$
for some permutation $\sigma\in {\sf Sym}(n)$.
\end{lemme}

\begin{proof}
Let  $C$ be the compact convex subset of $\R^{n\times n}$ consisting of all {\em doubly stochastic matrices} $[a_{ij}] \in \R^{n\times n}$, i.e., satisfying 
$$
\sum_{i=1}^na_{ik}=\sum_{i=1}^na_{ki}=1
$$
for all $k$ and $0\leqslant a_{ij}\leqslant 1$. By the Birkhoff--von Neumann theorem (due to D. K\H{o}nig \cite{konig}), the extreme points of $C$ are the permutation matrices, that is, with entries $0$ and $1$.
Observe that the norm $\Norm{\sum_{i=1}^n{x_i}-\sum_{i=1}^n{y_i}}_{\AE}$ is the infimum of the  linear map $\nu \colon \R^{n\times n}\to \R$, given by
$$
\nu\big([a_{ij}]\big)=\sum_{i,j=1}^na_{ij}d(x_i,y_j),
$$
over the convex set $C$.  Therefore, the infimum must be attained at an extreme point  of $C$, i.e., at a $\{0,1\}$-valued matrix $[a_{ij}]$. 
Letting $\sigma(i)=j$ when $a_{ij}=1$, we see that $\sigma\in {\sf Sym}(n)$ and that $\Norm{\sum_{i=1}^n{x_i}-\sum_{i=1}^n{y_i}}_{\AE}=\sum_{i=1}^nd(x_i,y_{\sigma(i)})$.
\end{proof}


\section{The Kantorovich--Rubinstein norm}\label{sec:KR}
The setup of Section \ref{sec:Lipschitz} is well-suited to the case when one deals with atomic signed measures on metric spaces. However, eventually we shall consider applications to locally compact groups, for which another approach is warranted. 

In the following, if $X$ is a Polish space, we let $M(X)$ denote the vector space of finite signed Borel measures and $M^+(X)$ the cone of finite positive  Borel measures  on $X$. Let also $M_0(X)$ be the hyperplane of signed measures of mean $0$, that is,
$$
M_0(X)=\{\gamma\in M(X)\del \gamma(X)=0\}.
$$
Observe that, by the Hahn--Jordan decomposition theorem, these are exactly the signed measures $\eta$ that can be written as differences $\eta=\mu-\nu$ of measures $\mu, \nu\in M^+(X)$ with $\mu(X)=\nu(X)$. 

Suppose now that $X$ is a Polish space equipped with a compatible metric $d$. Thus, whereas $d$ induces the topology on $X$, we do not necessarily assume that $X$ is complete with respect to $d$.  Let also $M^+_d(X)$ denote the set of all $\mu\in M^+(X)$ for which
$$
\int_X d(x,y)\, d\mu(x)<\infty
$$
for all or, equivalently, for some $y\in X$ and set 
$$
{\sf KR}(X)=\Mgd{\mu-\nu}{\mu, \nu\in M^+_d(X)\;\;\&\;\; \mu(X)=\nu(X)}.
$$
Evidently, ${\sf KR}(X)$ is a linear subspace of $M_0(X)$.

\begin{defi}
We define the {\em Kantorovich--Rubinstein norm} of a signed measure $\gamma\in {\sf KR}(X)$ to be the quantity
$$
\norm\gamma_{\sf KR}=\sup_{L(\phi)\leqslant 1}\;\int_X\phi\, d\gamma,
$$
where the supremum is taken over all $1$-Lipschitz functions $\phi\colon X\to R$.
\end{defi}
Evidently $\norm\cdot_{\sf KR}$  defines a seminorm on the vector space ${\sf KR}(X)$. 
Remark also that, if we view an element $\xi \in \M X$ as a finitely supported signed measure on $X$, then $\M X\subseteq {\sf KR}(X)$. Furthermore, by Kantorovich duality, that is, Equation \ref{AE}, we have that
$$
\norm\xi_{\sf KR}=\sup_{L(\phi)\leqslant 1}\;\int_X\phi\, d\xi=\sup_{L(\phi)\leqslant 1}\;\sum_{x\in X}\xi(x)\phi(x)= \norm{\xi}_{\AE}
$$
for all $\xi \in \M X$, so the inclusion $\M X\subseteq {\sf KR}(X)$ is isometric. 

Recall that, if $\mu, \nu \in M^+(X)$ are two positive measures with $\mu(X)=\nu(X)$, a {\em coupling} of $\mu$ and $\nu$ is an element $\eta\in M^+(X^2)$,  such that
$$
\mu=\eta_1, \quad \nu =\eta_2
$$
where $\eta_1,\eta_2\in M^+(X)$ denote the marginals of $\eta$, that is, $\eta_1(A)=\eta(A\times X)$ and $\eta_2(B)=\eta(X\times B)$ for all Borel sets $A,B$. For example, when $\mu\neq 0$, one may take $\eta=\frac1{\mu(X)}(\mu\times \nu)$. Let $\Pi(\mu,\nu)$ denote the convex set of all such couplings.

We now have a version of Kantorovich duality in this setting.

\begin{thm}[Kantorovich duality]
Suppose $X$ is a Polish space with a compatible metric $d$ and let $\mu,\nu\in M^+_d(X)$ with $\mu(X)=\nu(X)$. Then
$$
\norm{\mu-\nu}_{\sf KR}=\inf_{\eta\in \Pi(\mu,\nu)}\int_{X^2}d(x,y)\,d\eta(x,y).
$$
\end{thm}
Indeed, by homogeneity, it suffices to consider the case when $\mu$ and $\nu$ are probability measures on $X$. In this case, the equality follows from  Theorem 5.10 and Particular Case 5.16  in \cite{villani}.

Observe one important consequence of this duality, namely, whereas the right hand side initially appears to depend on the exact pair $(\mu,\nu)$ of measures, the left hand side only depends on their difference $\mu-\nu$. 

When $\mu$ and $\nu$ are probability measures,  the expression 
\begin{equation}\label{Wasserstein}
{\sf W}(\mu,\nu)=\inf_{\eta\in \Pi(\mu,\nu)}\int_{X^2}d(x,y)\,d\eta(x,y)
\end{equation}
is the so-called {\em Wasserstein distance}\footnote{This is also known as the {\em Kantorovich} or {\em Earth movers distance}, see the bibliographical notes to Chapter 6 in \cite{villani}.}  between $\mu$ and $\nu$. Of course, being an extension of the Arens--Eells norm to measures, the Kantorovich--Rubinstein norm $\norm{\mu-\nu}_{\sf KR}$ and thus the Wasserstein distance ${\sf W}(\mu,\nu)$ gauges the cost of an optimal transport between the two measures viewed respectively as  distributions of resources and consumption.

The Wasserstein distance is an actual metric on the space of probability measures in $M^+_d(X)$  (see, e.g., Chapter 6 \cite{villani}) and thus the Kantorovich--Rubinstein norm is a norm and not just a seminorm. This is due to the fact that the infimum is attained in Equation \ref{Wasserstein}, i.e., that there is an optimal transport $\eta$ between $\mu$ and $\nu$ (Theorem 4.1 \cite{villani}). Thus, if  ${\sf W}(\mu,\nu)=0$, then
$$
\int_{X^2}d(x,y)\,d\eta(x,y)=0
$$
and so $\eta$ is concentrated on the diagonal $\{(x,x)\del x\in X\}$, whereby $\mu=\nu$.

\begin{lemme}
Let $X$ be a Polish space with a compatible metric $d$. Then $\M X$ is dense in ${\sf KR}(X)$ and hence
$$
\AE X=\ov{{\sf KR}(X) }^{    \norm{\cdot}_{\sf KR}   }.
$$
\end{lemme}
Even when $(X,d)$ is an uncountable compact metric space, ${\sf KR}(X)$ is not complete for the norm (see  Proposition 2.3.2 \cite{weaver}) and thus ${\sf KR}(X)$ does not provide us with  a description of the elements of $\AE X$.

\begin{proof}
It suffices to show that every probability measure $\mu \in M^+_d(X)$ is arbitrarily close in the Wasserstein metric to a finitely supported probability measure. So let $\mu \in M^+_d(X)$ and $\eps>0$ be given. Fix also some $y\in X$. Because $\int_X d(x,y)\, d\mu(x)<\infty$, there is some diameter $r>0$ so that 
$$
\int_{\{x\in X\,|\, d(x,y)>r\}}d(x,y)\,d\mu(x)<\eps/3.
$$
Letting 
$$
\nu=\mu\begr_{\{x\in X\,|\, d(x,y)\leqslant r\}}+   { \mu( \{x\in X\,|\, d(x,y)>r\})}  \cdot  {\delta_y}  
$$
we obtain a probability measure with bounded support so that ${\sf W}(\mu,\nu)<\eps/3$. By tightness of Borel probability measures on Polish spaces (Theorem 7.1.4 \cite{dudley}), we may find a compact set $K\subseteq X$ so that $\nu(X\setminus K)<\eps/3r$. Letting $\rho=\nu\begr_K+\nu(X\setminus K)\cdot \delta_y$, we thus have a compactly supported Borel probability measure so that ${\sf W}(\nu,\rho)<\eps/3$. Pick now a partition of $K\cup \{y\}=\bigcup_{i=1}^nF_i$ into closed subsets of diameter  $<\eps/3$ and choose $x_i\in F_i$. Then
$$
\tau=\sum_{i=1}^n\rho(F_i)\cdot\delta_{x_i}
$$
is finitely supported and ${\sf W}(\rho, \tau)\leqslant \eps/3$. So ${\sf W}(\mu,\tau)<\eps$ as required.
\end{proof}

We thus have isometric inclusions of normed vector spaces
$$
\M X\subseteq {\sf KR}(X)\subseteq \AE X,
$$
where already $\M X$ is dense in $\AE X$. Moreover, the dual of these spaces is $\go{Lip}\, X$ and the duality with ${\sf KR}(X)$ is given by
$$
\langle \phi,\gamma\rangle =\int_X\phi\, d\gamma.
$$


\section{Mean operators and canonical invariant subspaces}\label{mean-operators}
Let $G$ be a group and $(V,\pi)$ is an isometric Banach $G$-module. We define the following two $G$-invariant closed linear subspaces
$$
V^G=\{v\in V\del \pi(g)v=v \text{ for all }g\in G\}
$$
and 
$$
V_G=\ov\spa\{v-\pi(g)v\del v\in V\;\&\; g\in G\}.
$$
Indeed, to see that $V_G$ is invariant, just note that, if $f,g\in G$ and $v\in V$, then
$$
\pi(f)(v-\pi(g)v)=\pi(f)v-\pi(fgf\inv)\pi(f)v\in V_G
$$
and so $V_G$ is spanned by a $G$-invariant subset and is therefore itself invariant.

\begin{lemme}\label{M13}
With respect to the dual actions of $G$ on $V^*$ and $V^{**}$, we have 
$$
(V^*)^G=(V_G)^\perp= ((V^{**})_G)_\perp
$$
and 
$$
V^G=((V^*)_G)_\perp= ((V^{**})^G)_{\perp\perp}.
$$
Furthermore, if $W$ is any $G$-invariant subspace of $V$, then either $W\subseteq V^G$ or $W\cap V_G\neq \{0\}$.
\end{lemme}

\begin{proof}
Indeed, for $\phi\in V^*$, we have
\maths{
\phi\in (V^*)^G
&\equi \a g\in G\;\a v\in V \;\phi(v-\pi(g)v)=(\phi-\pi(g)^*\phi)(v)=0\\
&\equi \phi\in (V_G)^\perp.
}
Similar arguments show  the equalities  $V^G=((V^*)_G)_\perp$, $(V^*)^G=((V^{**})_G)_\perp$ and $(V^{**})^G=((V^*)_G)^\perp$.
From this it finally follows that
$$
V^G=((V^*)_G)_\perp=(((V^*)_G)^\perp)_{\perp\perp}=((V^{**})^G)_{\perp\perp}.
$$
For the last part, note that, if $W\not\subseteq V^G$, we can pick some $w\in W\setminus V^G$ and find $g\in G$ so that $\pi(g)w\neq w$. Then $0\neq w-\pi(g)w\in W\cap V_G$.
\end{proof}

Suppose now that $G$ is a topological group, that $(V,\pi)$ is an isometric Banach $G$-module and fix a mean $\go r$ on the algebra ${\sf LUCB}(G)$. Then, for all $v\in V$ and $\phi\in V^*$, the coefficient function
$$
g\in G\mapsto \phi(\pi(g)v)\in \R
$$
belongs to ${\sf LUCB}(G)$. We may therefore define the associated {\em mean operator}
$$
V\overset R\longrightarrow V^{**}
$$
by letting
$$
\langle \phi,Rv\rangle=\go r\big(\phi(\pi(\cdot)v)\big)
$$
for all $v\in V$ and $\phi\in V^*$.

For the following, we will need an easy consequence of the Hahn--Banach hyperplane separation theorem (see Theorem 2.4.7 \cite{pedersen} for the exact statement used).

\begin{lemme}\label{hahn}
Suppose $C$ is a non-empty convex subset of a Banach space $V$ and let $w\in V^{**}$ belong to the $w^*$-closure of $C$. Then
$$
\norm w\geqslant \inf_{v\in C}\norm v.
$$
\end{lemme}

\begin{proof}
Without loss of generality, $\eta=\inf_{v\in C}\norm v>0$.  So, let $B$ be the open ball in $V$ of radius $\eta$ centred at the origin. Then $B$ and $C$ are disjoint non-empty convex sets with $B$ open  and so, by the hyperplane separation theorem, there is $\phi\in V^*$ so that, for all $v\in B$, 
$$
\phi(v)< \inf_{u\in C}\phi(u).
$$
Replacing $\phi$ by $\frac \phi{\norm\phi}$, we may suppose that $\norm\phi=1$. Thus, 
$$
\inf_{v\in C}\norm v=\eta=\sup_{v\in B}\phi(v)\leqslant \inf_{v\in C}\phi(v)\leqslant \langle \phi, w\rangle \leqslant \norm w
$$
as claimed.
\end{proof}

\begin{prop}\label{prop:meanoper}
Suppose $G$ is a topological group, $(V,\pi)$ is an isometric Banach $G$-module and that $\go r$ is a mean on $G$. 
Then the associated mean operator $V\overset R\longrightarrow V^{**}$ and adjoint $V^{***}\maps {R^*}V^*$ satisfy
$$
Rv\in \ov{   \pi(\Delta G)v} ^{w^*},
$$
$$
R^*\phi\in \ov{   \pi(\Delta G)^*\phi} ^{w^*},
$$
and 
$$
\inf_{\beta\in \Delta G}\norm{\pi(\beta)v}\leqslant \norm{Rv}\leqslant \norm v
$$
for all $v\in V$ and $\phi\in V^*$. Furthermore, $R=I$ on $V^G$, whereas $R^*=I$ on $(V^*)^G$.

It follows that
$$
{\sf ker}\, R\;\subseteq\; \Mgd{v\in V}{0\in \ov{   \pi(\Delta G)v} ^{\norm\cdot}} \;\subseteq\; V_G
$$
and that
$$
V^G\;\subseteq \; {\sf ker}\,(I-R)\;\subseteq\; R\inv(V)\;=\;\Mgd{v\in V}{Rv\in \ov{   \pi(\Delta G)v} ^{\norm\cdot}},
$$
whereby 
$$
R\inv(V)\;\maps {I-R}\; V_G.
$$
\end{prop}

\begin{proof}
Suppose towards a contradiction  that $Rv\notin \ov{   \pi(\Delta G)v} ^{w^*}$  for some $v\in V$ and pick by the Hahn--Banach separation theorem for the $w^*$-topology on $V^{**}$ some $\phi\in V^*$ such that 
$$
\langle \phi, Rv\rangle < \inf_{\beta\in \Delta G}\langle \phi, \pi(\beta)v\rangle\leqslant \phi(\pi(\cdot)v).
$$ 
Because $\go r$ is a mean, it follows that 
$$
\langle \phi, Rv\rangle < \inf_{\beta\in \Delta G}\langle \phi, \pi(\beta)v\rangle\leqslant \go r(\phi(\pi(\cdot)v))=\langle \phi,Rv\rangle,
$$
which is absurd. 

Similarly, if $\phi\in V^*$ satisfies $R^*\phi\notin \ov{   \pi(\Delta G)^*\phi} ^{w^*}$, there is some $v\in V$ such that 
$$
\langle \phi, Rv\rangle=\langle R^*\phi, v\rangle <  \inf_{\beta\in \Delta G}\langle \pi(\beta)^*\phi, v\rangle=\inf_{\beta\in \Delta G}\langle \phi, \pi(\beta)v\rangle\leqslant \phi(\pi(\cdot)v), 
$$ 
given again the absurd conclusion
$$
\langle \phi, Rv\rangle < \inf_{\beta\in \Delta G}\langle \phi, \pi(\beta)v\rangle\leqslant \go r(\phi(\pi(\cdot)v))=\langle\phi,Rv\rangle.
$$

Thus, for all $v\in V$, $Rv\in \ov{   \pi(\Delta G)v} ^{w^*}$ and so,  by Lemma \ref{hahn}, we find that $\norm{Rv}\geqslant \inf_{\beta\in \Delta G}\norm{\pi(\beta)v}$. Also, since $\go r$ is a mean, we have that $\langle \phi,Rv\rangle=\go r(\phi(\pi(\cdot)v))\leqslant \norm{\phi}\norm{v}$ for all $v$ and $\phi$, whereby $\norm R\leqslant 1$, which gives us  the inequality
$$
\inf_{\beta\in \Delta G}\norm{\pi(\beta)v}\leqslant \norm{Rv}\leqslant \norm v.
$$
It follows that, if $v\in {\sf ker}\,R$, then $\inf_{\beta\in \Delta G}\norm{\pi(\beta)v}=0$ and so $0\in \ov{   \pi(\Delta G)v} ^{\norm\cdot}$.

Note also that, if $0\in \ov{\pi(\Delta G)v}^{\norm\cdot}$, we may choose a sequence of elements $\beta_n\in \Delta G$ so that $\pi(\beta_n)v\Lim{n} 0$, whereby
$$
v=\lim_n(v-\pi(\beta_n)v)\in V_G.
$$
This establishes the inclusions ${\sf ker}\, R\subseteq \Mgd{v\in V}{0\in \ov{   \pi(\Delta G)v} ^{\norm\cdot}} \subseteq V_G$.

Observe now that, if either $v\in V^G$ or $\phi\in(V^*)^G$, we have 
$$
\langle R^*\phi,v\rangle=\langle \phi,Rv\rangle=\go r\big(\phi(\pi(\cdot)v)\big)=\go r\big(\phi(v)\big)=\langle \phi,v\rangle.
$$
So $R=I$ on $V^G$, whereas $R^*=I$ on $(V^*)^G$.  In particular, we have that $V^G\subseteq  {\sf ker}\,(I-R)\subseteq R\inv (V)$.

For the last part, observe that, if $Rv\in V$ for some $v\in V$, then actually 
$$
Rv\in V\cap  \ov{   \pi(\Delta G)v} ^{w^*}= \ov{   \pi(\Delta G)v} ^{\norm\cdot} 
$$
and so also
$$
(I-R)v=v-Rv\in V_G.
$$
\end{proof}

Observe that, for every $w\in V$ and $g\in G$, we may let  
$$
\beta_n=\frac{1+g+\cdots+g^{n-1}}n\in \Delta G
$$
and note that
$$
\pi(\beta_n)(w-\pi(g)w)=\frac{w-\pi(g^n)w}n\,\Lim{n}\,0,
$$
which shows that $w-\pi(g)w \in \Mgd{v\in V}{0\in \ov{   \pi(\Delta G)v} ^{\norm\cdot}}$. Nevertheless, this does not show that the set $\Mgd{v\in V}{0\in \ov{   \pi(\Delta G)v} ^{\norm\cdot}}$ equals the subspace $V_G$, since the former may not be a closed linear subspace. In fact, as can be seen from the results of Section \ref{decompo}, this characterises amenability of $G$.



\section{Decompositions for amenable groups}\label{decompo}
Assume now that we are dealing with an amenable topological group $G$ and an isometric Banach $G$-module $(V,\pi)$. In this case, we can fix a mean $\go m$ on ${\sf LUCB}(G)$ that is invariant with respect to the action $G\acts\rho{\sf LUCB}(G)$ and let  $V\overset M\longrightarrow V^{**}$ be the associated mean operator, given by
$$
\langle \phi,Mv\rangle=\go m\big(\phi(\pi(\cdot)v)\big).
$$
As a mean operator, Proposition \ref{prop:meanoper} immediately applies, but naturally $\rho$-invariance gives us more.

The next theorem has antecedents in the work of  S. Kakutani \cite{kakutani}, K. Yosida \cite{yosida}, L. Alaoglu and G. Birkhoff \cite{alaoglu}, R. J. Nagel \cite{nagel} and others. In particular, the complementation of $(V^*)^G$ in $V^{*}$ for the particular case when $G$ is either locally compact or balanced\footnote{We recall that a topological group is {\em balanced} or {\em SIN} if it has a neighbourhood basis at the identity consisting of conjugacy invariant sets. In the setting of metrisable topological groups, this is equivalent to admitting a compatible two-sided invariant metric, sometimes denoted as {\em tsi}.} is demonstrated by Cúth and Doucha in \cite{cuth}.

\begin{thm}\label{thm:alaoglu}
Suppose $G$ is an amenable topological group with $\rho$-invariant mean $\go m$ and  $(V,\pi)$ is an isometric Banach $G$-module.  Then the mean operator $V\overset M\longrightarrow V^{**}$  is $G$-invariant, that is, $M\pi(g)=M$ for all $g\in G$, whereby ${\sf ker}\,M=V_G$
and the adjoint operator
$$
V^{***}\overset{M^*}\longrightarrow V^*
$$
is a projection onto the $G$-invariant subspace $(V^*)^G$.

Moreover, 
$$
\norm{Mv}=\inf_{\beta\in \Delta G}\norm{\pi(\beta)v}
$$
and, for any finite subset $F\subseteq V_G$,
$$
\inf_{\beta\in \Delta G}\;\max_{v\in F}\;\Norm{\pi(\beta)v}=0.
$$
\end{thm}

\begin{proof}
Observe that, for all $g\in G$, $v\in V$ and $\phi\in V^*$, we have
$$
\langle \phi,M(\pi(g)v)\rangle=\go m\big(\phi(\pi(\cdot\,g)v)\big)=\go m\big(\rho(g)\big[\phi(\pi(\cdot)v)\big]\big)=\go m\big(\phi(\pi(\cdot)v)\big)=\langle \phi,Mv\rangle.
$$
It follows that $M\pi(g)=M$ for all $g\in G$, whereby $V_G\subseteq {\sf ker}\,M$  and hence ${\sf ker}\,M=V_G$ by Proposition \ref{prop:meanoper}.

To see that the adjoint $V^{***}\overset{M^*}\longrightarrow V^*$ maps into $(V^*)^G =(V_G)^\perp$, observe that,
for all $v^{***}\in V^{***}$ and $w\in V_G$, we have
$$
\langle M^*v^{***}, w\rangle=\langle v^{***}, Mw\rangle=\langle v^{***}, 0\rangle=0.
$$
Also, by Proposition \ref{prop:meanoper}, $M^*$ is the identity on $(V^*)^G$ and therefore is a projection onto $(V^*)^G$.

Applying Proposition \ref{prop:meanoper} and using the $G$-invariance of $M$, we find that
$$
\norm{Mv}=\inf_{\beta\in \Delta G}\norm{M(\pi(\beta)v)}\leqslant\inf_{\beta\in \Delta G}\norm{\pi(\beta)v}\leqslant \norm{Mv},
$$
giving us $\norm{Mv}=\inf_{\beta\in \Delta G}\norm{\pi(\beta)v}$.

Now, suppose $v_1,\ldots, v_n\in V_G$ and $\eps>0$. Find then $\beta_1\in \Delta G$ so that $\norm{\pi(\beta_1)v_1}<\eps$ and note that, because $V_G$ is $G$-invariant, also $\pi(\beta_1)v_2\in V_G$.  Choose therefore $\beta_2\in \Delta G$ so that $\norm{\pi(\beta_2\beta_1)v_2}<\eps$ and note again that $\pi(\beta_2\beta_1)v_3\in V_G$. Continuing in this manner, we eventually produce $\beta_1,\ldots, \beta_n\in \Delta G$, so that
$$
\norm{\pi(\beta_n\cdots\beta_1)v_i}\leqslant \norm{\pi(\beta_i\cdots\beta_1)v_i}<\eps
$$
for all $i=1,\ldots,n$. This shows that $\inf_{\beta\in \Delta G}\;\max_{v\in F}\;\Norm{\pi(\beta)v}=0$ for any finite subset $F\subseteq G$.
\end{proof}

\begin{cor}\label{cor:alaoglu}
Under the assumptions of Theorem \ref{thm:alaoglu}, we have the following direct sum decomposition
$$
M\inv(V)\;=\;\Mgd{v\in V}{Mv\in \ov{   \pi(\Delta G)v} ^{\norm\cdot}}\;=\; V_G\oplus {\sf ker}\,(I-M)
$$
and so the two subspaces $V_G$ and $V^G$ form a direct sum in $V$.
\end{cor}

\begin{proof}
By Proposition \ref{prop:meanoper} and Theorem \ref{thm:alaoglu}, we have that 
$$
M\inv(V)\maps{I-M}V_G={\sf ker}\,M \subseteq M\inv(V).
$$
Thus, as $I-M$ acts as the identity on ${\sf ker}\,M$, we see that it is a projection of $M\inv(V)$ onto the subspace ${\sf ker}\,M$ and that $M$ is the complimentary projection of $M\inv(V)$ onto ${\sf ker}\,(I-M)$. Using that ${\sf ker}\,M=V_G$ and ${\sf ker}\,(I-M)\supseteq V^G$, the last statement follows.
\end{proof}

If $(V,\pi)$ is a strictly convex Banach $G$-module, then the conclusion of Corollary \ref{cor:alaoglu} can be refined. Indeed, suppose $Mv=v$. Then, since also $\norm v=\norm{Mv}=\inf_{\beta\in \Delta G}\norm{\pi(\beta)v}$, we find that $v$ is an element of minimal norm in the convex set $ \ov{   \pi(\Delta G)v} ^{\norm\cdot}$ and so, by strict convexity, we have that $\{v\}=\ov{   \pi(\Delta G)v} ^{\norm\cdot}$, i.e., $v\in V^G$. Thus, in this case, we find that $ {\sf ker}\,(I-M)=V^G$ and so
$$
M\inv(V)=V_G\oplus V^G.
$$

Part of the conclusion of Theorem \ref{thm:alaoglu} can alternatively be reformulated using the Banach space theore\-tical notion of {\em local complementation}, which is worth reviewing. 
\begin{defi}\label{defi:loc compl}
Suppose $V$ is a Banach space and $W$ is a closed linear subspace. Then $W$ is {\em locally complemented} in $V$ if there is a constant $C\geqslant 1$ such that, for all finite-dimensional linear subspaces $F\subseteq V$ and  $\eps>0$, there is a linear operator $F\maps T W$ with $\norm T\leqslant C$ and
$$
\Norm{ v-Tv}<\eps\norm v
$$
for all $v\in W\cap F$.
\end{defi}

Local complementability was first studied by H. Fakhoury \cite{fakhoury} under a different name and reintroduced by N. Kalton \cite{kalton} with the current terminology. The following set of equivalences, largely due to Fakhoury, can be found as Lemma 3.2 and Theorem 3.5 in \cite{kalton}.
\begin{thm}\label{loc compl char}
The following conditions are equivalent for a closed linear subspace $W$ of a Banach space $V$. 
\begin{enumerate}
\item $W$ is locally complemented in $V$,
\item for some constant $C'$ and all closed linear subspaces $W\subseteq Z\subseteq V$ with ${\sf codim}_ZW<\infty$, there is a linear projection
$Z\maps PW$ with $\norm P\leqslant C'$,
\item $W^{**}$ is complemented in $V^{**}$ under its natural embedding,
\item the short exact sequence 
$$
0\to W^\perp \to V^*\to W^* \to 0
$$
splits linearly, that is, there is a bounded linear operator $W^*\maps L V^*$ such that 
$$
\langle w,L\phi\rangle =\langle w,\phi\rangle
$$ 
for all $w\in W$ and $\phi\in W^*$, 
\item for some constant $C''$ and all compact operators $W\maps K Z$ with values in any Banach space $Z$, there is a compact extension $V\maps{\tilde K}Z$ with $\norm{\tilde K}\leqslant C''\norm{K}$.
\end{enumerate}
\end{thm}

We then have the following corollary of Theorem \ref{thm:alaoglu}.

\begin{cor}\label{cor: loc compl}
Suppose $G$ is an amenable topological group and  $(V,\pi)$ is an isometric Banach $G$-module. Then $V_G$ is locally complemented in $V$.
\end{cor}

\begin{proof}
By Theorem \ref{thm:alaoglu}, $ (V_G)^\perp=(V^*)^G$ is complemented in $V^*$ and so the short exact sequence 
$0\to (V_G)^\perp \to V^*\to (V_G)^* \to 0$
splits linearly. This verifies Theorem \ref{loc compl char}(4) and so $V_G$ is locally complemented in $V$.
\end{proof}



\section{Skew amenability and $G$-equivariant projections}\label{sec:skew}
Our definition of amenability for topological groups, namely the existence of a right-invariant mean on ${\sf LUCB}(G)$, is based on the result of Rickert (Theorem 4.2 \cite{rickert}) that this coincides with the existence of fixed points for all continuous affine actions on non-empty compact convex subsets of  locally convex topological vector spaces. However, as advocated by V. Pestov in \cite{pestov}, there is another curious alternative of amenability, which codifies other interesting consequences of amenability of locally compact groups. A systematic study of this notion has been done by K. Juschenko and F. M. Schneider \cite{juschenko}.
\begin{defi}
A topological group $G$ is {\em skew-amenable} if there is a $\lambda$-invariant mean on ${\sf LUCB}(G)$,  that is, a positive linear functional $\go n\colon {\sf LUCB}(G)\to \R$ such that $\norm{\go n}=1$, $\go n(1) = 1$ and $\go n\big(\lambda(g)\phi\big)=\go n(\phi)$ for all $g\in G$ and $\phi\in {\sf LUCB}(G)$.
\end{defi}

Whereas amenability and skew amenability coincide in the categories of locally compact (see Theorem 2.2.1 \cite{greenleaf}) and balanced topological groups, they do not coincide in general. Indeed, the group ${\rm Aut}(\Q,<)$ of all order-preserving permutations of $\Q$ is amenable, but fails to be skew-amenable. Similarly, solving a problem of K. Juschenko, V. Pestov and F. M. Schneider, N. Ozawa showed that, if $\ku R$ denotes the hyperfinite $II_1$-factor, then the unitary group $\ku U(\ku R)$ is skew-amenable, but fails to be amenable. Thus, outside the category of locally compact groups, neither notion implies the other.

In the following, we shall briefly consider decompositions of isometric Banach $G$-modules for skew-amenable groups $G$. As can be seen, the consequences of skew amenability are rather different from those of amenability.

\begin{thm}\label{thm:alaoglu2}
Suppose $G$ is a skew-amenable topological group with $\lambda$-invariant mean $\go n$ and  $(V,\pi)$ is an isometric Banach $G$-module.  Then the mean operator $V\overset N\longrightarrow V^{**}$  maps into $(V^{**})^G$, whereas the the adjoint operator $V^{***}\maps {N^*} V^*$ is $G$-invariant, that is, $N^*\pi(g)^{***}=N^*$ for all $g\in G$. In particular, $V^G={\sf ker}\,(I-N)$.

It follows that
$$
N\inv(V)\;=\;\Mgd{v\in V}{Nv\in V^G\cap \ov{   \pi(\Delta G)v} ^{\norm\cdot}}\;=\; {\sf ker}\,N\oplus V^G
$$
and that $(V^*)^G$ and $(V^*)_G$ form a direct sum in $V^*$.
\end{thm}

\begin{proof}
Observe first that, for all $v\in V$, $\phi\in V^*$ and $g\in G$, 
\maths{
\langle \phi,\pi(g)^{**}Nv\rangle
&=\langle N^*\pi(g)^{*} \phi,v\rangle\\
&=\langle\pi(g)^{*} \phi,Nv\rangle\\
&=\go n\big(\phi(\pi(g\,\cdot)v)\big)\\
&=\go n\big(\lambda(g\inv)\big(\phi(\pi(\cdot)v)\big)\big)\\
&=\go n\big(\phi(\pi(\cdot)v)\big)\\
&=\langle \phi,Nv\rangle.
}
which shows that $\pi(g)^{**}N=N$ and $N^*\pi(g)^*=N^*$. In particular, ${\sf im}\, N\subseteq (V^{**})^G$ and hence, by Proposition \ref{prop:meanoper},
$$
N\inv(V)\;=\;\Mgd{v\in V}{Nv\in V^G\cap \ov{   \pi(\Delta G)v} ^{\norm\cdot}}.
$$

It follows that 
$$
N\inv(V)\maps N V^G
$$ 
and so, since by Proposition \ref{prop:meanoper} $N=I$ on $V^G$, we see that $N$ is a projection of $N\inv(V)$ onto $V^G$. As also $V^G\subseteq {\sf ker}\,(I-N)\subseteq N\inv(V)$, we find that $V^G={\sf ker}\,(I-N)$ and $N\inv(V)={\sf ker}\,N\oplus V^G$.

By the $G$-invariance of $N^*$, we have  $(V^*)_G\subseteq {\sf ker}\,N^*$. And, by Proposition \ref{prop:meanoper}, $
(V^*)^G\subseteq {\sf ker}(I-N^*)$.
As  ${\sf ker}\,N^*$ and ${\sf ker}(I-N^*)$ form a direct sum, so do the subspaces $(V^*)_G$ and $(V^*)^G$.
\end{proof}

\begin{thm}\label{G-equivariant}
Let $G$ be a topological group and $(V,\pi)$ an isometric Banach $G$-module. Assume also that $G$ is both amenable and skew-amenable. Then there exists a $G$-equivariant projection
$$
V^*\maps P(V^*)^G.
$$
\end{thm}

\begin{proof}
Because $G$ is both amenable and skew-amenable, we may define the two mean operators $V\maps{M,N}V^{**}$. Consider now the restriction of the adjoints to the subspace $V^*\subseteq V^{***}$,
$$
V^*\maps {N^*}V^*\maps{M^*}V^*.
$$
By Theorem \ref{thm:alaoglu2}, $N^*$ acts as the identity on $(V^*)^G$ and is $G$-invariant, i.e., $N^*\pi(g)^*=N^*$ for all $g\in G$. Similarly, by Theorem \ref{thm:alaoglu}, $M^*$ is a projection onto $(V^*)^G$. It thus follows that the composition $P=M^*N^*$ is $G$-invariant and acts as the identity on $(V^*)^G$. Furthermore, as ${\sf im}\, P=(V^*)^G$, we find that $P^2=P$. So $P$ is a $G$-invariant projection onto $(V^*)^G$ and therefore
$$
\pi(g)^*P=P=P\pi(g)^*.
$$
So $P$ is also $G$-equivariant.
\end{proof}

As the properties of amenability and skew amenability coincide in the two categories of locally compact and balanced topological groups, Theorem \ref{G-equivariant} in particular applies to amenable locally compact and to balanced topological groups. However, as the next example shows, it has a wider range of applications.

\begin{exa}[An amenable and skew-amenable Polish group that is neither locally compact nor balanced]
Let $\Gamma$ be a denumerable locally finite group and consider the semidirect product 
$$
G=\Gamma\ltimes \Z^\Gamma,
$$
induced by the right shift action $\Gamma\curvearrowright \Z^\Gamma$, that is, $(x^\gamma)_\sigma=x_{\sigma\gamma}$.
Thus, the multiplication in $G$ is given by
$$
(\gamma,x)\cdot(\sigma,y)=(\gamma\sigma, x\cdot y^\gamma)
$$
for $\gamma,\sigma\in \Gamma$ and $x,y\in \Z^\Gamma$. Because $\Gamma$ acts continuously on $\Z^\Gamma$, $G$ is a Polish topological group when equipped with the product topology of $\Gamma$ and $\Z^\Gamma$. Since the open subgroup $\Z^\Gamma$ is not locally compact, neither is $G$.  Also, if $x\in \Z^\Gamma$ is defined as 
$$
x_\gamma=\begin{cases}
1 &\text{if }\gamma=1\\
0&\text{if }\gamma\neq1,
\end{cases}
$$
then there is a sequence of conjugates of the non-identity element $(1,x)$ tending to the identity in $G$. Hence, $G$ is not balanced.

Let now $\Gamma_1\leqslant \Gamma_2\leqslant \ldots$ be an exhaustive chain of finite subgroups of $\Gamma$. Then the
$$
G_n=\Gamma_n\ltimes \Z^\Gamma=\{(\gamma,x)\in \Gamma\ltimes \Z^\Gamma\del \gamma\in \Gamma_n\}
$$
form an exhaustive chain of open subgroups of $G=\Gamma\ltimes \Z^\Gamma$. Furthermore, the subsets of the form
$$
V_{F,n}=\{(1,x)\in \Gamma_n\ltimes \Z^\Gamma\del x_\sigma=0 \text{ for all }\sigma\in F\Gamma_n\},
$$
where $F\subseteq \Gamma$ is any finite set, form a neighbourhood basis at the identity in the subgroup $\Gamma_n\ltimes \Z^\Gamma$ consisting of conjugacy invariant sets. It thus follows that each $G_n=\Gamma_n\ltimes \Z^\Gamma$ is balanced. 

Furthermore, as $\Z^\Gamma$ is amenable and has finite index in $\Gamma_n\ltimes \Z^\Gamma$, we see that $\Gamma_n\ltimes \Z^\Gamma$ is amenable. It thus follows that the union
$$
\Gamma\ltimes \Z^\Gamma=\bigcup_{n=1}^\infty \Gamma_n\ltimes \Z^\Gamma
$$
is amenable.

Since each $\Gamma_n\ltimes \Z^\Gamma$ is balanced and amenable, it is also skew-amenable. Thus, to see that $\Gamma\ltimes \Z^\Gamma$ is skew-amenable, it suffices to note that, by \cite[Lemma 6.4]{juschenko}, if $G$ is a topological group that can be written as the union of a chain $G_1\leqslant G_2\leqslant $ of skew-amenable subgroups, then $G$ is skew-amenable. 
\end{exa}



\section{Affine isometric actions and cohomology}\label{sec:affine}
We  recall a few facts about affine isometric group actions on Banach spaces. Namely, suppose $G$ is a topological group and $G\overset\alpha\curvearrowright V$ is a continuous  action by affine isometries on a Banach space $V$. Then there is a unique continuous isometric linear action $G\overset\pi\curvearrowright V$ and a  continuous map $G\overset b\longrightarrow V$ satisfying
\begin{equation}\label{cocycle}
\alpha(g)v=\pi(g)v+b(g)
\end{equation}
for all $g\in G$ and $v\in V$. Conversely, if $(V,\pi)$ is an isometric Banach $G$-module, we denote by $Z^1(G,\pi)$ the vector space of all continuous maps $G\overset b\longrightarrow V$  satisfying the above formula (\ref{cocycle}) defines an action $\alpha$ of $G$ by affine isometries on $V$. Specifically, $Z^1(G,\pi)$ may be described as the collection of all continuous $b$ satisfying the cocycle equation 
$$
b(gf)=\pi(g)b(f)+b(g)
$$
for all $g,f\in G$. Elements of $Z^1(G,\pi)$ are {\em cocycles} associated with $\pi$. Observe that $Z^1(G,\pi)$ is closed under taking linear combinations and thus is a vector space.

We now define the {\em coboundary operator} $V\maps\partial Z^1(G,\pi)$ by the formula
$$
(\partial v)(g)=v-\pi(g)v
$$
for all $v\in V$ and $g\in G$. That $\partial v\in Z^1(G,\pi)$ follows from the computation
$$
(\partial v)(gf)=v-\pi(gf)v=  \pi(g)\big(v-\pi(f)v\big)+\big(v-\pi(g)v\big)=\pi(g)(\partial v)(f) +(\partial v)(g).
$$
Cocycles of the form $\partial v$ are {\em coboundaries} and $B^1(G,\pi)={\sf im}\, \partial$ denotes the subspace of these. The quotient vector space 
$$
H^1(G,\pi)= Z^1(G,\pi)/  {B^1(G,\pi)}
$$
is the {\em $1$-cohomology} of the Banach $G$-module $(V,\pi)$. 

Each element  $b\in Z^1(G,\pi)$ is a map from $G$ to $V$ and thus  $Z^1(G,\pi)$ can be viewed as a subset of $\prod_{g\in G}V$. Equipping the latter with the Tychonoff product topology and  $Z^1(G,\pi)$ with the subspace topology,  $Z^1(G,\pi)$ becomes a locally convex topological vector space. The topological closure $\ov {B^1(G,\pi)}$ of $B^1(G,\pi)$ in  $Z^1(G,\pi)$ is then a closed linear subspace and the quotient
$$
\ov H^1(G,\pi)= Z^1(G,\pi)/ \ov {B^1(G,\pi)}
$$
a topological vector space. Elements of $\ov {B^1(G,\pi)}$ are called {\em almost coboundaries} and $\ov H^1(G,\pi)$ is termed the {\em reduced $1$-cohomology} of the Banach $G$-module $(V,\pi)$. By $\ov b$, we denote the image of a cocycle $b\in Z^1(G,\pi)$ in $\ov H^1(G,\pi)$.

Observe that, if $G\overset\alpha\curvearrowright V$ is the affine isometric action associated with a cocycle $b\in Z^1(G,\pi)$, then, for any $v\in V$ and $g\in G$, we have 
$$
\Norm{v-\alpha(g)v}=\Norm{(\partial v)(g)-b(g)}.
$$
This implies that the point $v\in V$ is fixed under the action $G\overset\alpha\curvearrowright V$ if and only if $b=\partial v$. Thus, coboundaries correspond exactly to affine actions with fixed points. Similarly, the action has {\em almost fixed points}, meaning that, for all $g_1,\ldots, g_n\in G$ and $\eps>0$, there exists $v\in V$  such that $\max_i\norm{v-\alpha(g_i)v}<\eps$, if and only if $b\in \ov {B^1(G,\pi)}$.

If $G$ is any topological group, we denote by 
$$
{\sf Hom}(G,\R)
$$
the vector space of continuous group homomorphisms $G\to \R$. 
\begin{lemme}\label{operator Q}
Let $G$ be a topological group and $(V,\pi)$ be an isometric Banach $G$-module. 
Then we may define a linear operator 
$$
\ov H^1(G,\pi)\otimes (V^*)^G \overset {Q}\longrightarrow {\sf Hom}(G,\R)
$$
by setting 
$$
Q({\ov b}\otimes \phi)=\phi\circ b
$$
for all $b\in Z^1(G,\pi)$ and $\phi\in (V^*)^G$.
\end{lemme}

\begin{proof}
Let us first verify that the map 
$$
( b,\phi)\in Z^1(G,\pi)\times (V^*)^G \mapsto \phi\circ b \in {\sf Hom}(G,\R)
$$
is well-defined and bilinear. So suppose that $b\in Z^1(G,\pi)$ and $\phi\in (V^*)^G=(V_G)^\perp$. Observe first that $G\maps b V \maps \phi \R$, whence the composition $\phi\circ b$ is well-defined. Also, 
$$
\big(\phi\circ b\big)(gf)
=\phi\big(\pi(g)b(f)\big)+\phi\big(b(g)\big)
=\phi\big(b(f))+\phi\big(b(g)\big)=\big(\phi\circ b\big)(f)+\big(\phi\circ b\big)(g),
$$
so $\phi\circ b\in {\sf Hom}(G,\R)$. Bilinearity is obvious.

Observe also that, if $b\in  \ov {B^1(G,\pi)}$, $g\in G$ and $\eps>0$, we may pick some $v\in V$ so that
$$
\norm{b(g)-\partial v(g)}<\eps,
$$
whereby
\maths{
\big|(\phi\circ b)(g)\big|
&\leqslant \big|\phi\big(b(g)-\partial v(g)\big)\big|+\big|\phi\big(\partial v(g)\big)\big|\\
&\leqslant \norm \phi\cdot  \Norm{b(g)-\partial v(g)}+\big|\phi\big(v-\pi(g)v\big)\big|\\
&<\norm \phi\cdot\eps+0.
}
Because $\eps>0$ is arbitrary, we find that $\phi\circ b=0$. It thus follows that the bilinear map above factors through to a bilinear map 
$$
({\ov b},\phi)\in \ov H^1(G,\pi)\times (V^*)^G \mapsto \phi\circ b \in {\sf Hom}(G,\R),
$$
which, by the defining properties of the tensor product is the same as a linear operator 
$$
\ov H^1(G,\pi)\otimes (V^*)^G \overset {Q}\longrightarrow {\sf Hom}(G,\R).
$$
\end{proof}

\begin{lemme}\label{V_G invariant}
Suppose that $G$ is a topological group with ${\sf Hom}(G,\R)=\{0\}$ and that $(V,\pi)$ is an isometric Banach $G$-module. Then the subspace $V_G$ is invariant under all continuous affine isometric actions $G\overset\alpha\curvearrowright V$ with linear part $\pi$.
\end{lemme}

\begin{proof}
Suppose $b\in Z^1(G,\pi)$ is the cocycle defining $\alpha$ from $\pi$. Then by Lemma \ref{operator Q} we see that $\phi\circ b=0$ for all $\phi\in (V^*)^G=(V_G)^\perp$, whereby ${\sf im}\,b\subseteq V_G$. As $V_G$ is $\pi$-invariant, it follows that 
$$
\alpha(g)v=\pi(g)v+b(g)\in V_G
$$
for all $v\in V_G$ and $g\in G$.
\end{proof}

\begin{lemme}\label{dim bound}
For all $\psi\in {\sf Hom}(G,V/V_G)$, we have
\begin{equation}\label{dim bdd}
{\sf dim}\;\ov{\sf span}\,\psi[G]\leqslant {\sf dim}\;{\sf Hom}(G,\R).
\end{equation}
\end{lemme}

\begin{proof}
Assume that  $g_1,\ldots, g_n\in G$ are chosen such that the vectors $\psi(g_1), \ldots, \psi(g_n)$ are linearly independent in $V/V_G$. Then we may find functionals  $\phi_1,\ldots, \phi_n\in \big(V/V_G\big)^*$  satisfying 
$$
\langle \psi(g_j),\phi_i\rangle=\delta_{ij},
$$
whereby the group homomorphisms 
$$
\langle \psi(\cdot ), \phi_1\rangle,\ldots, \langle \psi(\cdot ),  \phi_n\rangle
$$ 
are linearly independent elements in ${\sf Hom}(G,\R)$.
\end{proof}

\begin{rem}[Affine actions of $\Delta G$]
If $(V,\pi)$ is a Banach $G$-module, we let $\R G\overset\pi\longrightarrow \ku L(V)$ be the unique extension to a representation of the group algebra. Suppose also that $b\in Z^1(G,\pi)$ is a cocycle and extend $b$ linearly to a map $\R G\overset b\longrightarrow V$. We then note that for $\sigma=\sum_i\sigma_ig_i\in \Delta G$ and $\tau=\sum_j\tau_jf_j\in \Delta G$, we have
\maths{
b(\sigma\tau)
&=b\Big(\sum_i\sigma_ig\cdot \sum_j\tau_jf_j\Big)\\
&=\sum_i\sum_j\sigma_i \tau_jb(g_if_j)\\
&=\sum_i\sum_j\sigma_i \tau_j\big(\pi(g_i)b(f_j)+b(g_i)\big)\\
&=\sum_i\sigma_i \Big(\pi(g_i)\Big(\sum_j\tau_jb(f_j)\Big)+\sum_j\tau_jb(g_i)\Big)\\
&=\sum_i\sigma_i \Big(\pi(g_i)b(\tau)+b(g_i)\Big)\\
&=\pi(\sigma)b(\tau)+b(\sigma).
}
It thus follows that the formula
$$
\alpha(\sigma)v=\pi(\sigma)v+b(\sigma)
$$
defines an action $\Delta G\overset \alpha\curvearrowright V$ of the semigroup $\Delta G$ by affine transformations on $V$. Observe further that, because, for $\lambda_i>0$ with $\sum_i\lambda_i=1$ and $g_i\in G$,
$$
\alpha\big(\sum_i\lambda_ig_i\big)v=\sum_i\lambda_i\alpha(g_i)v,
$$
we have that
$$
{\sf conv}\,\{\alpha(g)v\del g\in G\}=\{\alpha(\beta)v\del \beta\in \Delta G\}
$$
for all $v\in V$. 
\end{rem}

\begin{thm}\label{affine}
Suppose that $G$ is an amenable topological group with ${\sf Hom}(G,\R)=\{0\}$ and that $(V,\pi)$ is an isometric Banach $G$-module. Then the subspace $V_G$ is invariant under all continuous affine isometric actions $G\overset\alpha\curvearrowright V$ with linear part $\pi$. Moreover, for any finite subset $F\subseteq V_G$, 
$$
\inf_{\beta\in \Delta G}{\sf diam}\,\{\alpha(\beta)v\del v\in F\}=0.
$$
\end{thm}

\begin{proof}
Let $G\overset\alpha\curvearrowright V$ be a continuous affine isometric action with linear part $\pi$ and cocycle $b\in Z^1(G,\pi)$. By Lemma \ref{V_G invariant}, $V_G$ is $\alpha$-invariant. Moreover, if $F\subseteq V_G$ is any finite set and $\eps>0$, it follows from Theorem \ref{thm:alaoglu} that there is some $\beta\in \Delta G$ so that $\norm{\pi(\beta)v}<\eps/2$ for all $v\in F$. Thus,
\maths{
\Norm{\alpha(\beta)v-\alpha(\beta)w}
&=\Norm{\big(\pi(\beta)v+b(\beta)\big)-\big(\pi(\beta)w+b(\beta)\big)}\\
&=\norm{\pi(\beta)v-\pi(\beta)w}\\
&<\eps
}
for all $v,w\in F$.
\end{proof}

Furthermore, under the hypotheses of Theorem \ref{affine}, we find that any finite collection of $\alpha$-invariant convex subsets of $V_G$ almost intersect. More precisely, we have the following corollary.
\begin{cor}
For any finite collection $C_1,C_2,\ldots, C_n$ of non-empty convex $\alpha$-invariant subsets of $V_G$ and any $\eps>0$, there is some $v\in V_G$ so that 
$$
\max_i\;{\sf dist}(v,C_i)<\eps.
$$
\end{cor}

\begin{proof}
Pick $v_i\in C_i$ and note that, because $v_i\in V_G$, there is some $\beta\in \Delta G$ so that $\norm{\pi(\beta)v_i}<\eps$ for all $i$. Writing $\beta=\sum_{j=1}^k\lambda_jg_j$ for some $g_j\in G$ and $\lambda_j>0$, we have that
$$
\sum_{j=1}^k\lambda_j\alpha(g_j)v_i=\sum_{j=1}^k\lambda_j\pi(g_j)v_i+\sum_{j=1}^k\lambda_jb(g_j)=\pi(\beta)v_i+\sum_{j=1}^k\lambda_jb(g_j)
$$
for all $i$.
Therefore, for $v=\sum_{j=1}^k\lambda_jb(g_j)\in V_G$, we find that
$$
{\sf dist}(v,C_i)\leqslant \Norm {v-\sum_{j=1}^k\lambda_j\alpha(g_j)v_i}=\norm{\pi(\beta)v_i}<\eps
$$
for all $i$.
\end{proof}

\begin{thm}\label{skew affine}
Let $G\overset\alpha\curvearrowright V$ be a continuous action by affine isometries by a skew-amenable topological group $G$ on a Banach space $V$.  Assume also that orbits have finite diameter. Then, for all $g_1,\ldots, g_n\in G$ and $\eps>0$, there is $v\in V$ so that $\max_i\norm{v-\alpha(g_i)v}<\eps$. That is, the action $\alpha$ almost fixes  points on $V$.
\end{thm}

\begin{proof}
Fix $g_1,\ldots, g_n\in G$ and consider the $\ell_1$-sum of $n$ copies of $V$, that is, $\prod_{i=1}^nV$ with norm 
$$
\norm{(v_1,\ldots,v_n)}=\norm{v_1}+\cdots+\norm{v_n}
$$
and let
$$
C=\big\{\big(v-\alpha(g_i)v\big)_{i=1}^n\in \prod_{i=1}^nV\del v\in V\big\}.
$$
Because each $\alpha(g_i)$ is an affine map, $C$ is convex.

We claim that $0\in \ov C^{\norm\cdot}$. Indeed, if $0\notin \ov C^{\norm\cdot}$, we can by the Hahn--Banach separation theorem find some $\phi_1,\ldots,\phi_n\in V^*$ and $\eps>0$ so that
$$
\sum_i\phi_i\big(v-\alpha(g_i)v\big)\geqslant \eps
$$
for all $v\in V$. Because the $\alpha$-orbit of $0$ is bounded, it follows that the functions $G\maps{\psi_i} \R$ defined by
$$
\psi_i(f)=\phi_i\big(\alpha(f)0\big)
$$
are left-uniformly continuous and bounded, that is $\psi_i\in {\sf LUCB}(G)$. Furthermore,
$$
\sum_i\big(\psi_i-\lambda(g_i\inv)\psi_i\big)\geqslant \eps.
$$
Thus, if $\go n$ is a $\lambda$-invariant mean on ${\sf LUCB}(G)$, we find that
\maths{
\eps
&\leqslant \go n\Big(\sum_i\big(\psi_i-\lambda(g_i\inv)\psi_i\big)\Big)\\
&=\sum_i\Big(\go n(\psi_i)-\go n\big(\lambda(g_i\inv)\psi_i\big)\Big)\\
&=\sum_i\Big(\go n(\psi_i)-\go n\big(\psi_i\big)\Big)\\
&=0,
}
which is absurd. So $0\in \ov C^{\norm\cdot}$, showing that  $\alpha$ almost fixes points on $V$.
\end{proof}
Observe that another way of expressing Theorem \ref{skew affine} is by saying that, if $(V,\pi)$ is an isometric Banach $G$-module for a skew-amenable topological group $G$, then every bounded cocycle $b\in Z^1(G,\pi)$ belongs to $\ov{B^1(G,\pi)}$.

Let us also note that the condition that $b$ is bounded is essential. Indeed, let $G=\Z/2\Z\ltimes \R$ be the group of affine isometries of the real line $\R$. Then $G$ is both amenable and skew-amenable and has no non-zero homomorphisms to $\R$, whereas the tautological action on $\R$ does not have almost fixed points.



\section{Isometric group actions on metric spaces}\label{isom}
Suppose $G\curvearrowright X$ is a continuous isometric action by a topological group $G$ on a metric space $X$. We then have an induced linear action $G\overset \pi\curvearrowright\R  X$ defined by $\pi(g)\xi=\xi(g\inv \,\cdot\, )$ or, equivalently, by the formula
$$
\pi(g)\Big(\sum_{i=1}^na_ix_i\Big)=\sum_{i=1}^na_igx_i.
$$ 
Because $\M X$ is $\pi$-invariant and the action on $\M X$ is by isometries, we can extend  the action to $\AE X$ and thus obtain the isometric Banach $G$-module $(\AE X,\pi)$. 

Recall that
$$
\go {Lip}\,X={\sf Lip}\, X/\{\text{constant functions}\}
$$
and observe that the contragredient action on $\go{Lip}\,X$ is just
$$
\pi^*(g)\phi=\phi(g\inv \,\cdot\,).
$$
Thus, if $\phi\in {\sf Lip}\,X$ and $g\in G$, the equality  $\pi(g)\phi=\phi$ holds in $\go {Lip}\, X$ if and only if $\phi-\phi(g\,\cdot\,)$ is constant on $X$. It thus follows that
$$
(\go{Lip}\, X)^G=\{\phi\in \go{Lip}\, X\del \text{for all }g\in G,\text{ the function } \phi-\phi(g\,\cdot\,)\text{ is constant}\}.
$$

\begin{lemme}\label{R}
The formula 
$$
(R\phi)(g)=\text{the constant value of }\phi-\phi(g\;\cdot\,)
$$
defines a linear operator
$$
(\go{Lip}\, X)^G\overset R\longrightarrow {\sf Hom}(G,\R).
$$
\end{lemme}

\begin{proof}
For each $x\in X$, we observe that the formula
$$
b_x(g)=x-gx
$$
defines a cocycle $b_x\in Z^1(G,\pi)$. Indeed, 
$$
b_x(gf)=x-gfx=\pi(g)(x-fx)+(x-gx)=\pi(g)b_x(f)+b_x(g)
$$ 
for all $g,f\in G$. On the other hand, for all $x,y\in X$ and $g\in G$, we have 
$$
(b_x-b_y)(g)= (x-gx)-(y-gy)=(x-y)-\pi(g)(x-y)=\partial v(g),
$$
where $v=x-y\in \AE X$. It thus follows that the cocycles $b_x$ and $b_y$ are {\em cohomologous}, that is, induce the same cohomology class or element in $H^1(G,\pi)=Z^1(G,\pi)/B^1(G,\pi)$ and, a fortiori, the same element in $\ov H^1(G,\pi)$.

Observe now that, if $Q$ is the operator of Lemma \ref{operator Q}, the operator 
$$
\begin{tikzcd}
(\go{Lip}\, X)^G\arrow{rr}{{Q({\ov b}_x\otimes \,\cdot\,)}}      & &     {\sf Hom}(G,\R)   
\end{tikzcd}
$$
is independent of $x\in X$. Furthermore, for $\phi \in (\go{Lip}\, X)^G$ and $g\in G$, we have 
\maths{
Q({\ov b}_x\otimes \phi)(g)
&=\big(\phi\circ b_x\big)(g)\\
&=\phi(x)-\phi(gx)\\
&=\text{the constant value of }\phi-\phi(g\;\cdot\,)\\
&=(R\phi)(g).
}
So $R=Q({\ov b}_x\otimes \,\cdot\,)$ for any choice of $x\in X$.
\end{proof}

Define  the following $G$-invariant closed linear subspace of $\AE X$,
$$
DX=\ov\spa\{x-gx\del x\in X\;\&\; g\in G\}
$$
and note that 
$$
(\AE X)_G\;\subseteq \; DX
$$
and 
\maths{
(DX)^\perp
&=\{\phi\in \go{Lip}\, X\del \phi\text{ is constant on every orbit of }G\curvearrowright X\}\\
&={\sf ker}\,R.
}
Because $(\go{Lip}\, X)^G=\big((\AE X)_G\big)^\perp$, we thus have an exact sequence of linear maps
$$
0\;\longrightarrow\;(DX)^\perp \;\overset I\longrightarrow\; (\go{Lip}\, X)^G\;\overset R\longrightarrow\; {\sf Hom}(G,\R).
$$
Observe also that
$$
{\sf codim}_{DX}(\AE X)_G={\sf codim}_{((\AE X)_G)^\perp}(DX)^\perp={\sf codim}_{(\go{Lip}\, X)^G}(DX)^\perp={\sf rank}\, R,
$$
so ${\sf rank}\, R$ measures the difference between $DX$ and $(\AE X)_G$.

\begin{lemme}\label{unbounded metrics1}
Suppose $G\curvearrowright X$ is a continuous isometric action by a topological group $G$ on a metric space $X$ and let $(\go{Lip}\, X)^G\overset R\longrightarrow {\sf Hom}(G,\R)$ be the operator of Lemma \ref{R}. Then, if ${\sf rank}\,R<\infty$, the subspace
$$
\{\phi\in \go{Lip}\, X\del \phi\text{ is constant on every orbit of }G\curvearrowright X\}
$$
is complemented in $(\go{Lip}\,X)^G$. 
\end{lemme}

\begin{proof}
By the above exact sequence, we see that ${\sf rank}\,R<\infty$ if and only if 
$$
({DX})^\perp=\{\phi\in \go{Lip}\, X\del \phi\text{ is constant on every orbit of }G\curvearrowright X\}
$$
has finite codimension in $(\go{Lip}\,X)^G$. The lemma now follows by noting that every closed linear subspace of finite codimension in a Banach space is complemented in the ambient space.
\end{proof}

\begin{lemme}\label{rank zero}
Suppose $G\curvearrowright X$ is a continuous isometric action by a topological group $G$ on a metric space $X$ so that either
\begin{enumerate}
\item ${\sf diam}_d(X)<\infty$,
\item ${\sf Hom}(G,\R)=\{0\}$,
\item any two orbit equivalent $x,y\in X$ can be transposed by some element of $G$, 
\item the set of  $g\in G$ for which
$$
\inf_{n\geqslant 1}\inf_{x\in X}\frac {d(x,g^nx)}n=0
$$ 
generate a dense subgroup of $G$.
\end{enumerate}
Then ${\sf rank}\,R=0$ and thus ${DX}=(\AE X)_G$ and 
$$
(\go{Lip}\,X)^G=\{\phi\in \go{Lip}\, X\del \phi\text{ is constant on every orbit of }G\curvearrowright X\}.
$$
\end{lemme}

\begin{proof}
Note first that (1) is a special case of (4). Also, the case (2) is obvious, since $R$ maps into ${\sf Hom}(G,\R)$.  On the other hand, in case (3), we note that, for all $x\in X$ and $g\in G$, there is some $h\in G$ so that $hx=gx$ and $hgx=h^2x=x$.   It thus follows that
$$
x-gx=\tfrac12\big( (x-hx) -\pi(h) (x-hx)\big) \in (\AE X)_G
$$
and so $DX=(\AE X)_G$.

In case (4),  for any $\phi\in (\go{Lip}\, X)^G$ and  a set of $g\in G$ generating a dense subgroup, we have 
$$
R\phi(g)=\inf_{n\geqslant 1}\frac {R\phi(g^n)}n=\inf_{n\geqslant 1}\inf_{x\in X}\frac {\phi x-\phi(g^nx)}n\leqslant L(\phi)\cdot \inf_{n\geqslant 1}\inf_{x\in X}\frac {d(x,g^nx)}n=0.
$$
Because $R\phi$ is a continuous group homomorphism, it follows that $R\phi=0$ and so ${\sf rank}\,R=0$, whereby ${DX}=(\AE X)_G$.
\end{proof}

For future reference, we also note two interesting cases in which the operator $R$ has finite rank by virtue of ${\sf Hom}(G,\R)$ being finite-dimensional.
\begin{lemme}\label{finite rank}
The requirement that $R$ has finite rank holds if
\begin{enumerate}
\item $G$ is {\em topologically finitely generated}, that is, contains a finitely generated dense subgroup,
\item $G$ is a connected locally compact Lie group. 
\end{enumerate}
\end{lemme}

Suppose again that $G\curvearrowright X$ is a continuous isometric action by a topological group $G$ on a metric space $X$. Because $G$ acts by isometries, the collection 
$$
X/\!\!/G=\big\{\ov {Gx}\,\big|\, x\in X\big\}
$$ 
of closures of orbits partitions $X$ into closed $G$-invariant sets.  Furthermore, the Hausdorff distance $d_H$ between closed subsets of $X$ then satisfies
$$
d_H\big(\ov{Gx},\ov{Gy}\big)=\inf_{g,f\in G}d(gx,fy)=\inf_{h\in G}d(hx,y)
$$
and thus defines a metric on $X/\!\!/G$. The next lemma is well-known (see, for example, Lemmas 4.2 and 4.3 \cite{cuth}), but we include the simple proof for completeness.

\begin{lemme}
Let $G\curvearrowright X$ be an isometric group action on a metric space. Then we have the following isometric isomorphisms of Banach spaces
$$
\AE X/{DX}\iso \AE(X/\!\!/G)
$$
and 
$$
\{\phi\in \go{Lip}\, X\del \phi\text{ is constant on every orbit of }G\curvearrowright X\}\iso \go{Lip}(X/\!\!/G).
$$
\end{lemme}

\begin{proof}
Define a linear operator 
$$
\M X\overset A\longrightarrow \M(X/\!\!/G)
$$
by
$$
(A\xi)(c)=\sum_{x\in c}\xi(x)
$$
and observe that
\maths{
\NORM{A\Big(\sum_{i=1}^nt_i(x_i-y_i)\Big)}_{\AE}
&=\NORM{\sum_{i=1}^nt_i\big(\ov{Gx_i}-\ov{Gy_i}\big)}_{\AE}\\
&\leqslant \sum_{i=1}^n|t_i|d_H\big(\ov{Gx_i},\ov{Gy_i}\big)\\
&\leqslant \sum_{i=1}^n|t_i|d(x_i,y_i),
}
which shows that $\norm A\leqslant 1$.  Conversely, suppose that $\zeta\in \M (X/\!\!/G)$, $\eps>0$ and write 
$$
\zeta=\sum_{i=1}^nt_i(b_i-c_i),
$$
so that $\norm \zeta_{\AE}=\sum_{i=1}^n|t_i|d_H(b_i,c_i)$. Choose $x_i\in b_i$ and $y_i\in c_i$ so that $d(x_i,y_i)<(1+\eps)d_H(b_i,c_i)$ for all $i$  and note that 
$$
A\Big(\sum_{i=1}^nt_i(x_i-y_i)\Big)=\zeta
$$
and 
$$
\NORM{\sum_{i=1}^nt_i(x_i-y_i)}_{\AE}\leqslant \sum_{i=1}^n|t_i|d(x_i,y_i)<(1+\eps)\norm\zeta_{\AE}.
$$
In other words, for all $\zeta\in \M(X/\!\!/G)$ and $\eps>0$, there is $\xi\in \M X$ with $\norm\xi_{\AE}<(1+\eps)\norm\zeta_{\AE}$ so that $A\xi=\zeta$. It follows  that $A$ extends uniquely  to a surjective bounded linear operator 
$$
\AE X\overset A\longrightarrow \AE(X/\!\!/G)
$$
so that
$$
{\sf ker}\, A=\ov{\M X\cap {\sf ker}\,A}=\ov\spa\, \{x-gx\del x\in X\;\&\; g\in G\}={DX}.
$$
Letting $\AE X\overset Q\longrightarrow \AE X/{DX}$ be the quotient map, we see that $A$ factors through the linear operator $\AE X/{DX}\overset J\longrightarrow \AE(X/\!\!/G)$ defined by $JQ=A$,
$$
\begin{tikzcd}
\AE X\arrow{dr}{Q}\arrow{rr}{A}      & &      \AE(X/\!\!/G)     \\
 &\AE X/{DX}\arrow{ur}{J} &\\
\end{tikzcd}
\vspace{-0.5cm}
$$
By the properties of $A$, the operator $J$ is an isometric isomorphism between $\AE X/{DX}$ and $\AE(X/\!\!/G)$.
Note also that
\maths{
\go{Lip}(X/\!\!/G) &=\AE(X/\!\!/G)^*\iso\big(\AE X/{DX}\big)^* \iso\big({DX}\big)^\perp\\
&=\{\phi\in \go{Lip}\, X\del \phi\text{ is constant on every orbit of }G\curvearrowright X\},
}
which proves the second part.
\end{proof}

Assuming moreover that ${\sf rank}\, R=0$, we furthermore have an isometric isomorphism
$$
\AE X/(\AE X)_G\iso  \AE(X/\!\!/G)
$$
and similarly
$$
\go{Lip}(X)^G\iso \go{Lip}(X/\!\!/G)
$$

\begin{lemme}
The isometric action $G\curvearrowright X$ is topologically transitive, that is, has a dense orbit, if and only if  $\AE X={DX}$.
\end{lemme}

\begin{proof}
Observe that the action is topologically transitive if and only if $X/\!\!/G$ is a singleton, which in turn happens if and only if $\AE(X/\!\!/G)=\{0\}$. The result thus follows from the isomorphism $\AE X/{DX}\iso \AE(X/\!\!/G)$.
\end{proof}

For the next lemma, we define the distance between two subsets of $X$ to be the infimum of distances between points in the two sets.
\begin{lemme}\label{no invariants}
Suppose $\xi ,\zeta\in \M X$ satisfy
$$
{\sf dist}\big({\sf supp}(\xi),  {\sf supp}(\zeta)\big)\geqslant \max\big\{{\sf diam}\, {\sf supp}(\xi), {\sf diam}\, {\sf supp}(\zeta)\big\}.
$$
Then
$$
\norm{\xi+\zeta}_{\AE}=\norm\xi_{\AE}+\norm\zeta_{\AE}.
$$
\end{lemme}

\begin{proof}
Write $\xi=a-b$ and $\zeta=c-d$, where $a,b,c,d$ are finite supported positive measures on $X$ so that ${\sf supp}(a)\cap {\sf supp}(b)=\tom$ and ${\sf supp}(c)\cap {\sf supp}(d)=\tom$. For a finitely supported measure $e$ on $X$, let also
$$
e(X)=\sum_{x\in X}e(x)
$$
be the total $e$-measure of $X$. By the definition of the Arens--Eells norm,  we can find finitely supported positive measures $a_i,b_i,c_i,d_i$ for $i=1,2$ so that
$$
a=a_1+a_2, \quad b=b_1+b_2, \quad c=c_1+c_2, \quad d=d_1+d_2,
$$
$$
a_1(X)=b_1(X), \quad a_2(X)=d_2(X), \quad c_1(X)=d_1(X), \quad c_2(X)=b_2(X) 
$$
and 
\maths{
\norm{\xi+\zeta}_{\AE}
&=\norm{(a+c)-(b+d)}_{\AE}\\
&=\norm{a_1-b_1}_{\AE}+\norm{a_2-d_2}_{\AE}+\norm{c_1-d_1}_{\AE}+\norm{c_2-b_2}_{\AE}.
}
Because $a_1(X)+a_2(X)=a(X)=b(X)=b_1(X)+b_2(X)$, we find that
$$
c_2(X)=b_2(X)=a_2(X)=d_2(X).
$$
It follows that
\maths{
\norm{\xi}_{\AE}+\norm{\zeta}_{\AE}
&=\norm{a-b}_{\AE}+\norm{c-d}_{\AE}\\
&\leqslant \norm{a_1-b_1}_{\AE}+\norm{a_2-b_2}_{\AE}+\norm{c_1-d_1}_{\AE}+\norm{c_2-d_2}_{\AE}\\
&\leqslant \norm{a_1-b_1}_{\AE}+a_2(X)\cdot {\sf diam}\, {\sf supp}(\xi)\\
&{}\quad+\norm{c_1-d_1}_{\AE}+c_2(X)\cdot {\sf diam}\, {\sf supp}(\zeta)\\
&\leqslant \norm{a_1-b_1}_{\AE}+a_2(X)\cdot {\sf dist}\big({\sf supp}(\xi),  {\sf supp}(\zeta)\big)\\
&{}\quad+\norm{c_1-d_1}_{\AE}+c_2(X)\cdot {\sf dist}\big({\sf supp}(\xi),  {\sf supp}(\zeta)\big)\\
&\leqslant \norm{a_1-b_1}_{\AE}+\norm{a_2-d_2}_{\AE}+\norm{c_1-d_1}_{\AE}+\norm{c_2-b_2}_{\AE}\\
&=\norm{\xi+\zeta}_{\AE}.
}
The reverse inequality is immediate from the triangle inequality.
\end{proof}

\begin{exa}[Non-complementation of $(\AE X)_G$ and $DX$ in $\AE X$]
Observe that, when the isometric action $G\curvearrowright X$ has unbounded orbits, it follows from Lemma \ref{no invariants} that, for all $\xi\in \AE X$,
$$
\sup_{g\in G}\norm{\xi-\pi(g)\xi}_{\AE}=2\norm{\xi}_{\AE}
$$
and so, in particular, $(\AE X)^G=\{0\}$. Note also that, by Lemma \ref{M13}, the only possible $G$-invariant complement of $(\AE X)_G$ or of $DX$ in $\AE X$ is the subspace $(\AE X)^G$.

However, if $X/\!\!/G$ contains at least two points, then 
$$
\AE X/{DX}\iso \AE(X/\!\!/ G)\neq \{0\},
$$ 
whence $(\AE X)_G\subseteq {DX}\subsetneqq\AE X$.
So, if $G\curvearrowright X$ has unbounded orbits and fails to be topologically transitive, then neither $(\AE X)_G$ nor $DX$ have a $G$-invariant complement in $\AE X$.
\end{exa}


\section{Isometric actions by amenable groups}\label{amen gps}
We now arrive at the core applications, which combines the decomposition results of Section \ref{decompo} with the observations of Section \ref{isom}.

The following result was proved by M. C\'uth and M. Doucha \cite{cuth} for isometric actions on  metric spaces with bounded orbits by amenable topological groups that, furthermore, are either locally compact or balanced.  In the more general context below, it is a direct consequence of Lemma \ref{unbounded metrics1} and Theorem \ref{thm:alaoglu}. 
\begin{thm}\label{thm:cuth1}
Suppose $G\curvearrowright X$ is a continuous isometric action by an amenable  topological group $G$ on a metric space $X$ and assume that $(\go{Lip}\, X)^G\overset R\longrightarrow {\sf Hom}(G,\R)$ has finite rank.
Then
$$
\{\phi\in \go{Lip}\, X\del \phi\text{ is constant on every orbit of }G\curvearrowright X\}
$$
is complemented in $\go {Lip}\, X$.
\end{thm}

\begin{proof}
It suffices to note that by Theorem \ref{thm:alaoglu} and Lemma \ref{unbounded metrics1} there are projections
$$
\go{Lip}\, X\overset P\longrightarrow(\go{Lip}\, X)^G\overset Q\longrightarrow ({DX})^\perp.
$$
As furthermore
$$
({DX})^\perp=\{\phi\in \go{Lip}\, X\del \phi\text{ is constant on every orbit of }G\curvearrowright X\},
$$
$QP$ is the requested projection. 
\end{proof}

\begin{rem}\label{rem:rank sero}
Observe that when ${\sf rank}\,R=0$, then $(\go{Lip}\, X)^G= ({DX})^\perp$ and the projection $P$ onto $(\go{Lip}\, X)^G$ provided by Theorem \ref{thm:alaoglu} has norm $1$. Thus, in this case, we find that 
$$
\{\phi\in \go{Lip}\, X\del \phi\text{ is constant on every orbit of }G\curvearrowright X\}
$$
is $1$-complemented in $\go {Lip}\, X$.
\end{rem}

Note that in Theorem \ref{thm:cuth1} the complement of the subspace 
$$
\{\phi\in \go{Lip}\, X\del \phi\text{ is constant on every orbit of }G\curvearrowright X\}
$$ 
need not in general be $G$-invariant. In order to ensure this, we need further assumptions, namely that $G$ is additionally skew-amenable, in which case we may additionally apply Theorem \ref{G-equivariant}.

\begin{thm}\label{thm:equivproj}
Suppose $G$ is an amenable and skew-amenable topological group and that $G\curvearrowright X$ is a continuous isometric action on a metric space $X$. Assume furthermore that either $X$ has finite diameter or that $G$ has no non-trivial continuous homomorphisms to $\R$. 
Then there is a $G$-invariant  contraction 
$$
\go {Lip}\, X \maps S\go{Lip}(X/\!\!/G)
$$
so that $(S\phi\big)\big(\ov{Gx}\big)=\phi(x)$ whenever $\phi$ is constant on every orbit of $G\curvearrowright X$.
\end{thm}

Our next result is similarly a reformulation and extension to potentially unbounded \'ecarts\footnote{An {\em écart}  or {\em pseudometric} $d$ on a set $X$ is the same as a metric except that possibly $d(x,y)=0$ for distinct $x,y\in X$.} of the main result of Schneider and Thom's paper \cite{thom}. The main import of their theorem is that it provides a Reiter or F\o lner type characterisation of amenability of topological groups beyond the locally compact ones. The passage to unbounded metrics however affords a new interesting formulation of amenability even for discrete amenable groups.
Nevertheless, in order to get the result for unbounded \'ecarts, we need to exclude examples such as $\R$ with the standard euclidean metric. Thus, we have the somewhat ironical  situation that the characterisation of amenability below does not apply to the prototypical examples of amenable groups such as $\Z$ and $\R$.

\begin{thm}\label{amen-char}
Suppose $d$ is a continuous left-invariant \'ecart on an amenable topological group $G$ and assume that either $d$ is bounded  or that $G$ has no non-trivial continuous homomorphisms to $\R$. 
Then, for every finite subset $E\subseteq G$, we have 
$$
\inf_{\beta\in \Delta G}\;\;\max_{g,f\in E}\;\;\norm{\beta g-\beta f}_{\AE}=0.
$$
Furthermore, for all $\eps>0$, there are $h_1,\ldots, h_n\in  G$ for which
$$
\max_{g,f\in E}\;\;\min_{\sigma\in {\sf Sym}(n)}\;\;\frac1n \sum_{i=1}^nd(h_ig,h_{\sigma(i)}f)<\eps.
$$
\end{thm}

\begin{proof}
Let $X=G/d$ be the metric space obtained from $G$ by identifying elements of $d$-distance $0$ and consider 
the isometric left-multiplication action $G\curvearrowright X$. For simplicity of notation, we identify elements of $G$ with their image in $G/d$. Set
$$
F=\{g-f\in \AE X\del g,f\in E\}
$$
and note that $F\subseteq {DX}$. By  Lemma \ref{rank zero},  $DX=(\AE X)_G$. Therefore, by Theorem \ref{thm:alaoglu},  we find that
$$
\inf_{\beta\in \Delta G}\;\;\max_{g,f\in E }\;\;\norm{\beta g-\beta f}_{\AE}=\inf_{\beta\in \Delta G}\;\;\max_{\xi \in F}\;\;\norm{\pi(\beta)\xi}_{\AE}=0.
$$
Now, for $\eps>0$ fixed, choose $\beta\in \Delta G$ so that $\max_{g,f\in E }\;\;\norm{\beta g-\beta f}_{\AE}<\eps$. By perturbing $\beta$, we may assume that $\beta$ is a rational convex combination of elements of $G$. This means that we may write 
$$
\beta=\frac 1n(h_1+\cdots+h_n)
$$
for some sequence $h_1,\ldots, h_n\in G$ (possibly with repetition among the $h_i$).  It thus follows from Lemma \ref{konig} that
$$
\min_{\sigma\in {\sf Sym}(n)}\;\frac 1n\sum_{i=1}^nd(h_ig,h_{\sigma(i)}f)=\norm{\beta g-\beta f}_{\AE}<\eps,
$$
which finishes the proof.
\end{proof}

Under the hypotheses of Theorem \ref{amen-char}, a straightforward application of Markov's inequality gives us  the following corollary.
\begin{cor}
For all finite $E\subseteq G$ and $\eps>0$, there are $h_1,\ldots, h_n\in G$ so that
$$
\max_{g,f\in E}\;\;\min_{\sigma\in {\sf Sym}(n)}\;\;{ \# \{i\del d(h_ig,h_{\sigma(i)}f)\geqslant\eps\}}<\eps n.
$$
\end{cor}

\begin{exa}[Failure of Theorem \ref{amen-char} for $\Z$ and $\R$]
Whereas Theorem \ref{amen-char} provides a very useful reformulation of amena\-bility for a very large class of topological groups, it is easy to see that fails for the  two amenable groups par excellence, namely, $\Z$ and $\R$. Indeed, consider the example of $\Z$ with the euclidean metric $d(x,y)=|x-y|$. To help keep vector space operations in $\AE \Z$ apart from algebra in $\Z$, we will denote by $\delta_x$ the Dirac functional at $x\in \Z$. So $\delta_x$ is veiwed as an element of the group algebra $\R \Z$. We claim that $E=\{0,1\}$ and $\eps=\frac 12$ provide a counter-example to Theorem \ref{amen-char} for $\Z$. Indeed,  for any $\beta=\sum_{i=1}^n\lambda_i\delta_{x_i}\in \Delta \Z$ with $x_i\in \Z$ and $\lambda_i>0$, $\sum_{i=1}^n\lambda_i=1$, we let $\phi\colon \Z\to \R$ be the $1$-Lipschitz function $\phi(x)=x$ and note that
\maths{
\norm{\beta\delta_1-\beta\delta_0}_{\AE}
&=\NORM{\sum_{i=1}^n\lambda_i\big(\delta_{x_i+1}-\delta_{x_i}\big)}_{\AE}\\
&\geqslant \sum_{i=1}^n\lambda_i(\phi({x_i+1})-\phi({x_i})\big)\\
& =\sum_{i=1}^n\lambda_i({x_i+1}-{x_i})\\
&=1.
}
So $\norm{\beta\delta_1-\beta\delta_0}_{\AE}>\eps$ for all choices of $\beta\in \Delta\Z$.
\end{exa}

\begin{exa}[The infinite dihedral group]\label{dihedral}
 The infinite dihedral group  $D_\infty=(\Z/2\Z)*(\Z/2\Z)$ can be represented as the group of all isometries of $\Z$, where $\Z$ is given the standard euclidean metric. Thus $D_\infty$ is generated by the translation $\tau(x)=x+1$ and the reflection $\sigma(x)=-x$. Alternatively, $D_\infty$ has the finite presentation
$$
D_\infty=\langle \tau,\sigma\del \sigma\tau=\tau^{-1}\sigma, \sigma^2=1\rangle.
$$
In particular, elements of $D_\infty$ can be uniquely written in the form $\tau^n$ or $\tau^n\sigma$ for some $n\in \Z$. Let $d$ be the left-invariant word metric associated with the finite generating set $S=\{\tau,\sigma\}$ and consider the associated Arens--Eells space $\AE D_\infty$. 

Suppose now that $E\subseteq D_\infty$ is a finite set and $\eps>0$. Without loss of generality, we can suppose that $E=\{\tau^n, \tau^n\sigma\del |n|\leqslant N\}$ and $\eps=\frac 1{2N}$ for some $N$. We then define $\beta\in \Delta D_\infty$ by
$$
\beta=\frac 1{4M+2}\Big(\sum_{n\in I}\tau^n +\sum_{n\in I}\tau^n\sigma\Big)
$$ 
where $I=[-M,M]$ for some $M$ so that $\frac{2N^2+2N}{4M+2}<\frac 1N$ and  claim that 
$$
\norm{\beta g-\beta f}_{\AE}<\eps
$$
for all $g,f\in E$.

To see this, it suffices to show that $\norm{\beta g-\beta}_{\AE}<2\eps=\frac 1N$ for all $g\in E$. So let $g\in E$ be given and suppose first that $g=\tau^k$ with $0\leqslant k\leqslant N$. Then
$$
\beta g=\frac 1{4M+2}\Big(\sum_{n\in (I+k)}\tau^n +\sum_{n\in (I-k)}\tau^n\sigma\Big)
$$
and so 
\maths{
(4&M+2)\cdot \norm{\beta g-\beta}_{\AE}\\
\leqslant& 
\NORM{
\sum_{n\in (I+k)\setminus I}\tau^n+\sum_{n\in (I-k)\setminus I}\tau^n\sigma-\sum_{n\in I\setminus (I+k)}\tau^n -\sum_{n\in I\setminus (I-k)}\tau^n\sigma
}_{\AE}\\
\leqslant& 
\sum_{n=1}^k\Norm{\tau^{M+n} -\tau^{M-k+n}\sigma}_{\AE}+\sum_{n=0}^{k-1}\Norm{\tau^{-M-k+n}\sigma -\tau^{-M+n}}_{\AE}\\
\leqslant& 
\sum_{n=1}^k\Norm{\tau^{k} -\sigma}_{\AE}+\sum_{n=1}^k\Norm{\sigma -\tau^{k}}_{\AE}\\
\leqslant &k\cdot (k+1)+k\cdot (k+1)\\
\leqslant &2N^2+2N.
}
It thus follows that $\norm{\beta g-\beta}_{\AE}\leqslant \frac{2N^2+2N}{4M+2}<\frac 1N$ as required. All other cases of $g=\tau^k$ or $g=\tau^k\sigma$ for $|k|\leqslant N$ are proved in a similar way.
\end{exa}

From Example \ref{dihedral}, we see that for the infinite dihedral group the almost translation-invariant averages $\beta\in \Delta D_\infty$ can be taken on the form 
$$
\beta=\frac 1{|F|}\sum_{f\in F}f
$$
for appropriate finite subsets $F\subseteq D_\infty$. In particular, this means that the Arens--Eells norm $\norm{\beta -\beta g}_{\AE}$ can be expressed as follows
$$
\norm{\beta -\beta g}_{\AE}=\frac 1{|F|}\sum_{f\in F\setminus Fg}d(f,\theta(f))
$$
for some bijection $\theta$ between $F\setminus Fg$ and $Fg\setminus F$. This formulation provides a more obvious generalisation of F\o lner's criterion and it is thus natural to wonder if this can always be achieved.

\begin{prob}
Suppose $\Gamma$ is a finitely generated amenable group with no non-zero homomorphisms to $\R$ and equip $\Gamma$ with its left-invariant word metric. Let $\eps>0$ and  $E\subseteq \Gamma$ be finite. Can the $\beta\in \Delta\Gamma$ in Theorem \ref{amen-char} always be taken to be of the form 
$$
\beta=\frac 1{|F|}\sum_{f\in F}f
$$
for some appropriate finite subset $F\subseteq \Gamma$?
\end{prob}

\section{Locally compact groups}\label{LCSC}
In the following, let $G$ be a locally compact second-countable group equipped with its left Haar measure. By $|A|$ and $\int f\,dx$ we denote respectively the Haar measure of a subset $A\subseteq G$ and the integral of a function $f\in L^1(G)$, where the latter denotes the space of real-valued integrable functions on $G$. We recall that $L^1(G)$ is a Banach algebra under the convolution product
$$
(f*h)(x)=\int f(y)h(y\inv x)\,dy.
$$

Assume now that $(V,\pi)$ is an isometric Banach $G$-module. Similarly to the extension of the representation $\pi$ from $G$ to the group algebra $\R G$, we may also promote $\pi$ to the Banach algebra $L^1(G)$. Indeed, this is done via the Pettis integral as follows. Fix $f\in L^1(G)$ and let $v\in V$ be given. Because the map $x\in G\mapsto \pi(x)v\in V$ is bounded and norm continuous, by Theorem A3.3 \cite{folland}, there is a unique vector in $V$, denoted $\int f(x)\pi(x)v\,dx$, so that
$$
\Big\langle \int f(x)\pi(x)v\,dx\,,\,\phi\Big\rangle=\int\langle f(x)\pi(x)v,\phi\rangle\,dx
$$
for all $\phi\in V^*$. Furthermore, by Theorem A3.3 \cite{folland}, 
$$
\NORM{\int f(x)\pi(x)v\,dx}\leqslant \sup_{x\in G}\norm{\pi(x)v}\cdot \int|f|\,dx=\norm v\cdot \norm f_{L^1}.
$$
We thus see that the map $v\mapsto \int f(x)\pi(x)v\,dx$ defines a linear operator on $V$ with norm bounded by $\norm f_{L^1}$. Denote this operator by $\pi(f)=\int f(x)\pi(x)\,dx$ and note that $\pi(f)$ depends linearly on $f$. Moreover, for all $f,h\in L^1(G)$, 
\maths{
\pi(f*h)
&=\iint f(y)h(y\inv x)\pi(x)\, dy\, dx\\
&=\iint f(y)h(z)\pi(yz)\, dz\, dy\\
&=\iint f(y)h(z)\pi(y)\pi(z)\, dz\, dy\\
&=\pi(f)\pi(h).
}
Summing up, we obtain a Banach algebra representation
$$
L^1(G)\overset \pi\longrightarrow\ku L( V)
$$
with $\norm{\pi(f)}\leqslant  \norm f_{L^1}$.

\begin{defi}
Suppose that $d$ is a compatible left-invariant metric on $G$. We denote by $L^{1}_{d,+,1}(G)$ the collection of all probability densities $f\in L^1(G)$ so that the associated measure $\mu$ given by $d\mu=f\,dx$ belongs to $M^+_d(G)$. That is,
$$
L^{1}_{d,+,1}(G)=\big\{f\in L^1(G)\del f\geqslant 0, \;\norm{f}_{L^1}=1, \;  \int d(x,1)f(x)\,dx<\infty\big\}.
$$
\end{defi}
Slightly abusing notation, we simply let $f\, dx$ denote the measure $\mu$ on $G$ with density $f$, that is so that $d\mu=f\,dx$. Furthermore, 
the Wasserstein distance can also be viewed as a distance on the space of densities  $L^1_{d,+,1}(G)$ by simply letting
$$
{\sf W}(f,h)=\norm{f\, dx-h\, dx}_{\sf KR}.
$$

As usual, we let $\pi$ denote the linear isometric action of $G$ on $\AE G$ induced from the left-multiplication action of $G$ in itself. By the above, this induces a Banach algebra representation $L^1(G)\overset \pi\longrightarrow\ku L( \AE G)$. For $x\in G$ and $f\in L^1(G)$, let also $R_xf=f(\;\cdot\; x)$ and recall that $\Norm{\Delta(x)\!\cdot\!R_xf}_{L^1}=\norm f_{L^1}$, where $\Delta\colon G\to \R$ denotes the modular function on $G$. 

\begin{lemme}\label{pi(f)}
Suppose $f\in L^1_{d,+,1}(G)$ and $y,z\in G$. Then
$$
\pi(f)(\delta_{y\inv}-\delta_{z\inv})=\Delta(y)(R_yf)\, dx-\Delta(z)(R_zf)\,dx.
$$
\end{lemme}

\begin{proof}
To see this, let $\phi\in \go {Lip}\; G=\AE G^*$. Then
\maths{
\big\langle \pi(f)(\delta_{y\inv}-\delta_{z\inv}), \phi \big\rangle
&=\int \big\langle f(x)\pi(x)(\delta_{y\inv}-\delta_{z\inv}),\phi\big\rangle \, dx\\
&=\int  f(x)\big\langle \delta_{xy\inv}-\delta_{xz\inv},\phi\big\rangle \, dx\\
&=\int  f(x)\phi(xy\inv)\,dx-\int f(x)\phi(xz\inv)\, dx\\
&=\Delta(y)\int  f(uy)\phi(u)\,du-\Delta(z)\int  f(uz)\phi(u)\,du\\
&=\int \phi\;\Delta(y)(R_{y}f)\,du-\int \phi\;\Delta(z)(R_{z}f)\,du\\
&=\big\langle \Delta(y)R_{y}f\,dx-\Delta(z)R_{z}f\, dx,\phi\big\rangle.\\
}
Because the dual $\go {Lip}\; G$ separates points in $\AE G$, the equality follows.
\end{proof}

\begin{thm}
Suppose $G$ is an amenable locally compact second-countable group and $d$ is a compatible left-invariant metric on $G$. Assume also that $G$ has no non-trivial continuous homomorphisms $G\to \R$. Then, for every compact subset $C\subseteq G$ and $\eps>0$, there is a compactly supported $f\in L^1_{d,+,1}(G)$ so that 
$$
{\sf W}(R_yf,R_zf)<\eps
$$
for all $y,z\in C$.
\end{thm}

\begin{proof}
Let us first remark that $\log \Delta\colon G\to \R$ is a continuous group homomorphism, which is therefore trivial. It thus follows that $\Delta\equiv 1$, i.e., that $G$ is unimodular. In particular,  for all $f\in L^1_{d,+,1}(G)$ and $y\in G$, we have
$\Norm{R_yf}_{L^1}=\norm{f}_{L^1}$ and
\maths{
 \int d(x,1)    R_yf(x)\,dx
 &= \int d(x,1)f(xy)\,dx\\
 &= \int d(zy\inv,1)f(z)\,dz\\
 &\leqslant \int d(zy\inv, z)f(z)\,dz+ \int d(z, 1)f(z)\,dz    \\
 &= \int d(y\inv, 1)f(z)\,dz+ \int d(z, 1)f(z)\,dz    \\
  &=  d(y\inv, 1)\norm{f}_{L^1}+ \int d(z, 1)f(z)\,dz    \\
  &<\infty,
}
which shows that $R_yf\in L^1_{d,+,1}(G)$.

Because $C$ is compact, we may pick a finite set $E\subseteq C$ so that 
$$
\sup_{x\in C}\inf_{y\in E}d(x\inv, y\inv)<\tfrac \eps4.
$$ 
Then, by Theorem \ref{amen-char}, choose some $\beta\in \Delta G$ so that 
$$
\norm{\pi(\beta)(\delta_{ y\inv}-\delta_{z\inv})}_{\sf KR}=\norm{\beta y\inv-\beta z\inv}_{\AE}<\frac\eps8
$$ 
for all $y,z\in E$. Let $F=\{\delta_{ y\inv}-\delta_{z\inv}\del y,z\in E \}$ and write $\beta=\sum_{i=1}^n\lambda_iu_i$ with $\lambda_i>0$, $\sum_{i=1}^n\lambda_i=1$ and $u_i\in G$. Pick compact neighbourhoods $U_i\subseteq G$ of $u_i$ so that 
$$
\norm{\pi(u_i)v -\pi(x)v }_{\sf KR}<\frac \eps8
$$
for all $x\in U_i$ and $v\in F$. Letting $f=\sum_{i=1}^n\frac{\lambda_i}{|U_i|}\chi_{U_i}$, where $\chi_{U_i}$ denotes the characteristic function of $U_i$, we see that $f\geqslant 0$, $\norm f_{L^1}=1$ and finally, because $f$ has compact support, that $\int d(x,1)f(x)\,dx<\infty$. Furthermore, for all $v\in F$ and $\phi\in \go{Lip}\,G$ with $L(\phi)\leqslant 1$, we have 
\maths{
\Big|\big\langle \pi(\beta)v-\pi(f)v, \phi \big\rangle\Big|
&\leqslant \sum_{i=1}^n\lambda_i\Big|\big\langle \pi(u_i)v, \phi \big\rangle-\frac1{|U_i|}\big\langle \pi(\chi_{U_i})v, \phi \big\rangle\Big|\\
&\leqslant \sum_{i=1}^n\lambda_i\Big|\big\langle \pi(u_i)v, \phi \big\rangle-\frac1{|U_i|}\int_{U_i}\big\langle \pi(x)v, \phi \big\rangle\, dx\Big|\\
&= \sum_{i=1}^n\lambda_i\Big|      \frac1{|U_i|}\int_{U_i}   \big\langle \pi(u_i)v- \pi(x)v, \phi \big\rangle\, dx\Big|\\
&\leqslant \sum_{i=1}^n\lambda_i      \frac1{|U_i|}\int_{U_i}  \Big| \big\langle \pi(u_i)v- \pi(x)v, \phi \big\rangle\Big|\, dx\\
&\leqslant \sum_{i=1}^n\lambda_i      \frac1{|U_i|}\int_{U_i} \norm{\pi(u_i)v -\pi(x)v }_{\sf KR}\,L(\phi)\,dx\\
&\leqslant \frac\eps8.
}
It thus follows that $\norm{\pi(\beta)v-\pi(f)v}_{\sf KR}\leqslant \frac \eps8$ for all $v\in F$ and hence, by unimodularity and Lemma \ref{pi(f)}, that
$$
{\sf W}(R_yf,R_zf)=\norm{R_yf\, dx-R_zf\,dx}_{\sf KR}=\norm{\pi(f)(\delta_{ y\inv}-\delta_{z\inv})}_{\sf KR}<\frac\eps4
$$
for all $y,z\in E$.

Assume now that $y,z\in C$ and pick some $y', z'\in E$ with $d(y,y')<\tfrac\eps4$ and $d(z,z')<\tfrac\eps4$. We then have
\maths{
W(R_yf,R_zf)
&\leqslant W(R_yf,R_{y'}f)+W(R_{y'}f,R_{z'}f)+W(R_{z'}f,R_zf)\\
&\leqslant \norm{\pi(f)(\delta_{ y\inv}-\delta_{(y')\inv})}_{\sf KR}+\tfrac\eps4+\norm{\pi(f)(\delta_{ (z')\inv}-\delta_{z\inv})}_{\sf KR}\\
&\leqslant \norm{\pi(f)}d({ y\inv},{(y')\inv})+\tfrac\eps4+\norm{\pi(f)}d((z')\inv,z\inv)\\
&<\eps, 
}
which finishes the proof. 
\end{proof}



\begin{thebibliography}{99}



\bibitem{alaoglu}L. Alaoglu and G. Birkhoff, {\em General ergodic theorems}, Ann. Math. 41 (1940), no. 2, 293--309.

\bibitem{arens}Richard Arens and James Eells, Jr., {\em On embedding uniform and topological spaces}, Pacific J. Math.
Vol. 6 (1956), No. 3, 397--403.




\bibitem{cuth}M. C\'uth and M. Doucha, {\em Projections in Lipschitz-free spaces induced by group actions}, Math. Nachr. 296 (2023), no. 8, 3301--3317.


\bibitem{deleuw}Karel de Leeuw, {\em Banach spaces of Lipschitz functions}, 
Studia Mathematica (1961)
Volume: 21, Issue: 1, page 55--66.

\bibitem{dudley}R. M. Dudley, {\em Real Analysis and Probability}, Cambridge Studies in Advanced Mathematics, Cambridge University Press, New York, 2002.



\bibitem{fakhoury}Hicham Fakhoury, {\em S\'elections lin\'eaires associ\'ees au th\'eor\`eme de Hahn--Banach}, J. Funct. Anal. 11 (1972) 436--452.
  
\bibitem{folland}G. B. Folland, {\em A course on Abstract Harmonic Analysis}, CRC Press, Boca Raton, FL, 1995.



\bibitem{folner}E. F\o lner, {\em On groups with full Banach mean value}, Math. Scand. 3 (1955), 243--254. 


\bibitem{greenleaf} Frederick P. Greenleaf, {\em Invariant means on topological groups and their applications}, Van Nostrand
Mathematical Studies 16, Van Nostrand Reinhold Co., New York-Toronto, Ont.-London, 1969.

\bibitem{juschenko}K. Juschenko and F. M. Schneider, {\em Skew-amenability of topological groups},  Comment. Math. Helv. 96 (2021), 805--851.


\bibitem{kakutani}Shizuo Kakutani, {\em Iteration of linear operations in complex Banach spaces}, 
Proc. Imp. Acad. 14 (8), 295--300, (1938).

\bibitem{kalton}Nigel J. Kalton, {\em Locally complemented subspaces and $\ku L^p$-spaces for $0<p<1$}, 
Math. Nachr. 115 (1984), 71--97.

\bibitem{kantorovich1}L. V. Kantorovich and G. Sh. Rubinstein, {\em On a functional space and certain extremal problems}, Dokl. Akad. Nauk. SSSR {\bf 115} (1957), 1058--1061.

\bibitem{kantorovich2}L. V. Kantorovich and G. Sh. Rubinstein, {\em On a space of completely additive functions}, Vestnik Leningrad Univ. Math. {\bf 13} (1958), 52--59. (Russian)

\bibitem{kantorovich3}L. V. Kantorovich and G. P. Akilov, {\em Functional Analysis} (second edition), Pergammon Press, New York, 1982. 

\bibitem{konig} D. K\H{o}nig, {\em Theorie der endlichen und unendlichen Graphen}, Akademische Verlagsgesellschaft M.B.H., Leipzig 1936.



\bibitem{megrelishvili}M. G. Megrelishvili, {\em Fragmentability and Continuity of Semigroup Actions}, Semigroup Forum, Vol. 57 (1998), 101--126.


\bibitem{moore}Justin T. Moore, {\em Amenability and Ramsey theory},  Fund. Math. 220 (2013), no. 3,  263--280.


\bibitem{nagel}Rainer J. Nagel {\em Mittelergodische Halbgruppen linearer Operatoren}, Ann. Inst. Fourier (Grenoble) 23 (1973), no. 4, 75--87.




\bibitem{ozawa} Narutaka Ozawa, {\em Amenability for unitary groups of simple monotracial C*-algebra},
Münster J. Math., to appear.   arXiv:2307.08267

\bibitem{pedersen}Gert K. Pedersen, {\em Analysis  NOW}, Graduate Texts in Mathematics, 118. Springer Verlag, New York, 1989.

\bibitem{pestov}Vladimir G. Pestov, {\em An amenability-like property of finite energy path and loop groups}, C. R. Math.
Acad. Sci. Paris 358 (2020), no. 11--12, pp. 1139--1155.


\bibitem{reiter}H. Reiter, {\em The convex hull of translates of a function in $L^1$}, J. London Math. Soc. 35 (1960), 5--16.


\bibitem{rickert}Neil W. Rickert, Amenable groups and groups with the fixed point property, Trans. Amer. Math.
Soc. 127 (1967), pp. 221--232.


\bibitem{thom2}Friedrich Martin Schneider and Andreas Thom, {\em Topological matchings and amenability},  Fundamenta Mathematicae 238 (2017), pp. 167--200.

\bibitem{thom}Friedrich Martin Schneider and Andreas Thom, {\em On F\o lner sets in topological groups}, Compositio Mathematica, Volume 154, Issue 7, July 2018, pp. 1333--1361.


\bibitem{villani}C\'edric Villani, {\em Optimal Transport, Old and New}, Volume 338, Grundlehren der mathematischen Wissenschaften, Springer Verlag, Berlin, 2009.

\bibitem{weaver}Nik Weaver, {\em Lipschitz algebras}, World Scientific Publishing Co., Inc., River Edge, NJ, 1999.

\bibitem{yosida}K\^{o}saku Yosida, {\em Mean ergodic theorem in Banach spaces}, 
Proc. Imp. Acad. 14 (8), 292--294 (1938).
\end{thebibliography}
\end{document}